\documentclass[reqno]{amsart}

\usepackage{amssymb,amsmath}
\usepackage{euscript}
\usepackage{graphicx}
\usepackage{tikz}
\usepackage{tikz-cd}
\usepackage{hyperref}
\usepackage{mathrsfs}
\usepackage{dsfont}
\usepackage{mathtools}

\newtheorem{thm}{Theorem}[section]
\newtheorem{cor}[thm]{Corollary}
\newtheorem{lem}[thm]{Lemma}
\newtheorem{prop}[thm]{Proposition}
\newtheorem{defin}[thm]{Definition}
\newtheorem{def-lem}[thm]{Definition-Lemma}
\newtheorem{conj}[thm]{Conjecture}

\theoremstyle{remark}

\newtheorem{rem}[thm]{Remark}

\newtheorem{example}[thm]{Example}

\newtheorem{assu}{Assumption}

\numberwithin{equation}{section}

\newcommand{\bbC}{\mathbb{C}}

\newcommand{\bbE}{\mathbb{E}}
\newcommand{\bbF}{\mathbb{F}}

\newcommand{\bbM}{\mathbb{M}}

\newcommand{\bbR}{\mathbb{R}}

\newcommand{\bbS}{\mathbb{S}}

\newcommand{\bbX}{\mathbb{X}}

\newcommand{\bbZ}{\mathbb{Z}}

\newcommand{\bfH}{\mathbf{H}}

\newcommand{\bfTheta}{\mathbf{\Theta}}

\newcommand{\scrF}{\mathscr{F}}

\newcommand{\scrP}{\mathscr{P}}
\newcommand{\scrU}{\mathscr{U}}

\newcommand{\calC}{\mathcal{C}}

\newcommand{\calH}{\mathcal{H}}
\newcommand{\calJ}{\mathcal{J}}

\newcommand{\calO}{\mathcal{O}}
\newcommand{\calP}{\mathcal{P}}

\newcommand{\frakD}{\mathfrak{D}}
\newcommand{\frakE}{\mathfrak{E}}
\newcommand{\frakF}{\mathfrak{F}}

\newcommand{\frakM}{\mathfrak{M}}

\newcommand{\frakQ}{\mathfrak{Q}}
\newcommand{\frakR}{\mathfrak{R}}

\newcommand{\frakX}{\mathfrak{X}}
\newcommand{\frakY}{\mathfrak{Y}}

\DeclareMathOperator{\codim}{codim}

\DeclareMathOperator{\crit}{Crit}

\DeclareMathOperator{\ind}{ind}

\DeclareMathOperator{\mas}{Mas}
\DeclareMathOperator{\ob}{ob}
\DeclareMathOperator{\param}{Par}

\newcommand{\aut}{\mathrm{Aut}}

\newcommand{\flow}{\mathrm{Flow}}
\newcommand{\fr}{\mathrm{fr}}

\newcommand{\ho}{\mathrm{Ho}}
\newcommand{\identity}{\mathrm{Id}}

\newcommand{\map}{\mathrm{Map}}

\newcommand{\reg}{\mathrm{reg}}
\newcommand{\sfr}{\mathrm{sfr}}
\newcommand{\spectra}{\mathrm{Sp}}

\newcommand{\std}{\mathrm{std}}

\renewcommand{\th}{\mathrm{Th}}

\newcommand{\unit}{\mathds{1}}

\newcommand{\abs}[1]{\lvert#1\rvert}

\title{Parameterized Lagrangian Floer homotopy}
\author{Kenneth Blakey}
\address{Department of Mathematics, MIT, 182 Memorial Drive, Cambridge, MA 02139, U.S.A.} 
\email{kblakey@mit.edu}

\author{Ciprian M. Bonciocat}
\address{Department of Mathematics, Stanford University, 450 Jane Stanford Way, Building 380,
Stanford, CA 94305-2125, USA.}
\email{ciprianb@stanford.edu}

\begin{document}

\begin{abstract}
We construct the Lagrangian Floer homotopy type, in the exact setting, as a spectrum parameterized over the moduli space of Maslov data. Our primary motivation for this construction is to provide stronger lower bounds for (possibly degenerate) Lagrangian intersections in plumbings of cotangent bundles.
\end{abstract}

\maketitle
\tableofcontents

\section{Introduction}
\subsection{Background and context}
The present article is about applying Floer homotopy theory to the study of (possibly degenerate) Lagrangian intersections; it was really spawned from intertwining the following two lines of thought (cf. Theorem \ref{thm:main} for the present article's main contribution.)

\subsubsection{Lagrangian intersections}
Let $(M,\omega)$ be a symplectic manifold together with two Lagrangians $L_i\subset M$, $i=0,1$. A fundamental problem in symplectic geometry is determining lower bounds for the cardinality of the intersection $L_0\cap L_1$.

\begin{conj}[Degenerate Arnol'd Conjecture]
Suppose $M$ is closed and the pair $(L_0,L_1)$ is Hamiltonian isotopic. Moreoever, suppose $L_i$ is: connected, closed, and $\omega\vert_{\pi_2(M,L_i)}=0$; then 
    \begin{equation}
    \abs{L_0\cap L_1}\geq\min\big\{\abs{\crit(f)}:f\in C^\infty(L_i)\big\},
    \end{equation}
where $\crit(f)$ is the set of critical points of $f$.
\end{conj}

In the case of the 0-section of a cotangent bundle of a closed smooth manifold, the degenerate Arnol'd conjecture is essentially known, cf. \cite{LS85}. When $(L_0,L_1)$ also satisfies the property that their intersection is transverse, \cite[Theorem 1]{Flo88b} proves a lower bound via the sum of the Betti numbers of $L_i$. If we allow the intersection of $(L_0,L_1)$ to be possibly degenerate, \cite[Theorem 3]{Hof88} and \cite[Theorem 1]{Flo89a} independently prove a lower bound via $c_{\bbZ/2}(L_i)$, i.e, the \emph{$\bbZ/2$-cuplength} of $L_i$: 
    \begin{equation}
    c_{\bbZ/2}(L_i)\equiv\inf_k\big\{k\in\bbZ_{\geq0}:\forall\alpha_1,\ldots,\alpha_k\in\widetilde{H}(L_i;\bbZ/2),\alpha_1\smile\cdots\smile\alpha_k=0\big\}.
    \end{equation}
Hirschi-Porcelli upgrade this $c_{\bbZ/2}(L_i)$ lower bound to a $c_\frakR(L_i)$ lower bound, where $\frakR$ is any ring spectrum, assuming certain $\frakR$-orientability hypotheses on the relevant compactified Floer-type moduli spaces, cf. \cite[Theorem 1.9]{HP22}. Finally, work of the first author, in the exact setting, goes beyond the case that $(L_0,L_1)$ is Hamiltonian isotopic (assuming $M$ admits a stable $\bbR$-polarization $\Lambda$ and $L_i$ admits a framed brane structure compatible with $\Lambda$, hence the \emph{(framed) Lagrangian Floer homotopy type} $\frakF^\Lambda$ exists; this is a spectral refinement of Lagrangian Floer cohomology which we review shortly). In particular, \cite[Theorem 1.11]{Bla24} shows we may use iterated compositions of Steenrod squares on 
    \begin{equation}
    H^*(\frakF^\Lambda;\bbZ/2)\cong HF^*(L_0,L_1;\bbZ/2)
    \end{equation}
to give a lower bound, provided each Steenrod square in the composition has degree at least $(\dim L_i-1)/2$. Moreover, \cite{Bla} extends this Steenrod square lower bound to a lower bound via iterated right coactions of the $\frakE$-generalized dual Steenrod algebra $\frakE_*\frakE$ on $\frakE_*\frakF^\Lambda$ (again with a degree condition).\footnote{Here, $\frakE$ is a ring spectrum such that $\frakE_*\frakE$ is a flat right $\pi_*\frakE$-module; examples include $KO$, $KU$, $MO$, $MU$, $MSp$, $\bbS$, and $H\bbF_p$.} \emph{Loc. cit.} also shows that the coaction of $\frakE_*\frakE$ on $\frakE_*\frakF^\Lambda$ is already captured by the dual $(\frakE\wedge\frakE)$-spectral quantum cap product (i.e., the right coaction of $(\frakE\wedge\frakE)_*L_i$ on $(\frakE\wedge\frakE)_*\frakF^\Lambda$) and for completeness proves a lower bound via the $\frakR$-spectral quantum caplength. In fact, the lower bounds in \cite{Bla,Bla24} show the guaranteed intersection points are of pairwise distinct action.

\subsubsection{Floer homotopy theory}
Cohen-Jones-Segal \cite{CJS95} introduced Floer homotopy idea as an idea to provide spectral refinements of various Floer (co)homologies. It was essentially sketched in \emph{loc. cit.}, with details given by Large \cite{Lar21}, how to spectrally refine Lagrangian Floer cohomology under the assumptions that:\footnote{Technically, Large constructed the Lagrangian Floer homotopy type under slightly stronger assumptions. However, the construction has been carried out under various assumptions similar to (and including) the ones listed here, cf. \cite{ADP24,Bla24,CK23,PS24a,PS24b}.}
\begin{itemize}
\item $M$ is a Liouville manifold;
\item $L_i$ is an exact Lagrangian which is either closed or cylindrical at infinity;
\item $M$ admits a stable $\bbR$-polarization $\Lambda$, i.e., a choice of factorization of a classifying map for the stable tangent bundle of $M$: 
    \begin{equation}
    \begin{tikzcd}
    & BO\arrow[dr,"(\cdot)\otimes_\bbR\underline{\bbC}"]\arrow[d,Rightarrow] & \\
    M\arrow[rr,"{TM}",swap]\arrow[ur,"\Lambda"] & \phantom{*} & BU
    \end{tikzcd}
    ;
    \end{equation}
\item $L_i$ admits a framed brane structure compatible with $\Lambda$, i.e., a choice of homotopy between a classifying map for the stable tangent bundle of $L_i$ and $\Lambda\vert_{L_i}$: 
    \begin{equation}
    \begin{tikzcd}[column sep = large]
    L_i \arrow[r,bend left = 20, "{TL_i}", ""{name=U,inner sep=1pt,below}]
    \arrow[r,bend right = 20,"\Lambda\vert_{L_i}"{below},swap, ""{name=D,inner sep=1pt,above}]
    & BO
    \arrow[Rightarrow, from=U, to=D, ""]
    \end{tikzcd}
    .
    \end{equation}
\end{itemize}
Under these assumptions, the compactified moduli spaces of Floer trajectories admit coherent stable framings and can be organized to obtain the (framed) Lagrangian Floer homotopy type $\frakF^\Lambda$. Note, an important point to observe is that, \emph{a priori}, $\frakF^\Lambda$ depends on $\Lambda$; this is a central theme of the present article.

There are now two digressions to make. Firstly, work of Abouzaid-Blumberg \cite{AB24} has rewritten the foundations of Floer homotopy theory using the language of stable $\infty$-categories. Since this will be better suited for the constructions in the present article, we will use these foundations. Secondly, the following assumptions suffice to construct the (spherically framed) Lagrangian Floer homotopy type; we will use these assumptions in the present article. Let $\scrP(L_0,L_1)$ be the space of paths starting on $L_0$ and ending on $L_1$; this space is homotopy equivalent to the homotopy fiber product of the inclusions $L_i\to M$. There is a map 
    \begin{align}\label{eq:intro}
    \scrP(L_0,L_1)&\to BO \\
    \gamma&\mapsto \big[T_{\gamma(1)}L_1\big]-\big[T_{\gamma(0)}L_0\big], \nonumber
    \end{align}
where the difference is defined after parallel transport along $\gamma$, whose composition with the map $BO\to BU$ is canonically trivially. Therefore, we have a canonical lift of \eqref{eq:intro} to $U/O$:
    \begin{equation}\label{eq:umodo}
    \begin{tikzcd}
    & U/O\arrow[d] \\
    \scrP(L_0,L_1)\arrow[ur]\arrow[r] & BO
    \end{tikzcd}
    ,
    \end{equation}
where we use the standard fiber sequence $O\to U\to U/O$. Finally, we consider the composition
    \begin{equation}\label{eq:assu}
    \scrP(L_0,L_1)\to U/O\xrightarrow{\sim}B(\Omega^\infty KO)\to B^2O\to B^2GL_1\bbS,
    \end{equation}
where the second map comes from real Bott periodicity, the third map is the projection, and the fourth map is the 2-fold delooping of the real $J$-homomorphism. 

\begin{assu}\label{assu:main}
Consider the following.
\begin{itemize}
\item $M$ is a Liouville manifold.
\item $L_i$ is an exact Lagrangian which is either closed or cylindrical at infinity.
\item $(L_0,L_1)$ admits Maslov data, i.e., a choice of null-homotopy $\theta$ of \eqref{eq:assu}.
\end{itemize}
\end{assu}

\begin{rem}
Clearly, a choice of stable $\bbR$-polarization and framed brane structures gives a choice of Maslov data. However, the converse need not hold, cf. \cite{AGLW25}.
\end{rem}

We will prove the folklore statement that, under Assumption \ref{assu:main}, the compactified moduli spaces of Floer trajectories admit coherent $\bbS$-orientations and can be organized to obtain the (spherically framed) Lagrangian Floer homotopy type $\frakF^\theta$; this is an $\frakR^\sfr$-module spectrum, where $\frakR^\sfr$ is a suitable ring spectrum conjecturally related to spherically framed bordism, cf. Remark \ref{rem:structuredflow} and Appendix \ref{appendix:sfrfr}. We do this by utilizing a new construction of the Bott isomorphism $U/O\xrightarrow{\sim}B(\Omega^\infty KO)$, due to the second author \cite{Bon25}, which is better suited for studying the index theory of Cauchy-Riemann operators with totally real boundary conditions. Again, $\frakF^\theta$ \emph{a priori} depends on $\theta$.

\subsubsection{Colliding worlds}
Consider the following toy example. 

\begin{example}\label{example:introtoy}
Let $M$ be the plumbing $T^*S^2\cup_{NS^1}T^*S^2$ of two $T^*S^2$'s along the normal bundle $NS^1\to S^1$ of the standard embedding $S^1\to S^2$ and consider the obvious (non-Hamiltonian isotopic) Lagrangian pair $(L_0,L_1)$, where $L_i$ is diffeomorphic to $S^2$, which intersects cleanly along $S^1$ (cf. Example \ref{example:plumbing} for the general construction). One should expect that we cannot deform $L_i$, via a compactly supported Hamiltonian deformation, so that $L_0$ intersects $L_1$ in a single (possibly degenerate) point; in fact, one should expect that we require at least 2 (possibly degenerate) intersection points of distinct action. Since we are currently considering the non-Hamiltonian isotopic case, we would like to find a Steenrod square or spectral quantum cap product which guarantees these 2 action levels must always exist. If we consider the ``standard'' stable $\bbR$-polarization $\Lambda_\std$ on $M$ together with the ``standard'' framed brane structure on $L_i$ (cf. the proof of Proposition \ref{prop:pullback} for the general construction), we have a homotopy equivalence $\frakF^{\Lambda_\std}\simeq\Sigma^\infty_+S^1$. In particular, there are no interesting Steenrod squares or spectral quantum cap products which provide a lower bound (note, the spectral quantum cap product is uninteresting since the embedding $S^1\to S^2$ is null-homotopic). However, there is a ``non-standard'' stable $\bbR$-polarization $\Lambda$ on $M$ together with a ``non-standard'' framed brane structure on $L_i$, obtained by twisting the standard structures by the non-trivial line bundle on $S^1$, which gives a homotopy equivalence $\frakF^{\Lambda}\simeq\Sigma^{-1}\bbR P^2$. Now, $\Sigma^{-1}\bbR P^2$ has the desired Steenrod square.
\end{example}

Example \ref{example:introtoy} shows that considering different homotopical input data to construct the Lagrangian Floer homotopy type can yield more geometric information. Note, this example really only used the fact that $\pi_0$ of the space of framing data for the Lagrangian Floer homotopy type of $(S^2,S^2)$ in $T^*S^2\cup_{NS^1}T^*S^2$ is non-trivial. If we consider the plumbing $T^*S^n\cup_{NS^k}T^*S^n$ along the standard embedding $S^k\to S^n$ ($k<n$), we are naturally led to the following more general strategy.

\subsection{Main result}\label{subsec:mainresult}
Let $\mas$ be the space of null-homotopies $\theta$ of \eqref{eq:assu} (i.e., the space of Maslov data) and $\mas^\theta$ the pointed space  $\mas,\theta\big)$. The aim of the present article is to construct the (spherically framed) Lagrangian Floer homotopy type as a parameterized spectrum, i.e., an $\infty$-functor 
    \begin{equation}
    \frakF^\sfr:\mas^\theta\to\mathrm{Mod}_{\frakR^\sfr},
    \end{equation}
where $\mathrm{Mod}_{\frakR^\sfr}$ is the stable $\infty$-category of $\frakR^\sfr$-module spectra. Our main application of this construction is the following.

\begin{thm}\label{thm:main}
Let $Q_i$ be two smooth manifolds together with an embedding of a closed smooth manifold $C\to Q_i$ such that the normal bundle $N_CQ_i\to C$ of $C$ in $Q_i$ is isomorphic to a fixed real vector bundle $NC\to C$. Consider the plumbing $M$ of $T^*Q_i$ along $NC$ and $(L_0,L_1)$ the obvious (non-Hamiltonian isotopic) Lagrangian pair, where $L_i$ is diffeomorphic to $Q_i$, which intersects cleanly along $C$, cf. Example \ref{example:plumbing}. Suppose in addition that $C$ is: connected, orientable, and positive-dimensional; then any compactly supported Hamiltonian deformation of $L_0$ intersects any compactly supported Hamiltonian deformation of $L_1$ in at least 2 (possibly degenerate) points of distinct action.
\end{thm}

\begin{rem}
We still consider the setup of the previous theorem. Of course, the work of Po\'zniak \cite{Poz94} (which essentially proves a spectral sequence computing the Lagrangian Floer cohomology of cleanly intersecting Lagrangians) immediately implies that any compactly supported Hamiltonian deformation of $L_0$ intersects any compactly supported Hamiltonian deformation of $L_1$ in at least 2 non-degenerate points. Theorem \ref{thm:main} shows that one cannot achieve less intersection points and/or action levels by allowing degenerate intersection points.
\end{rem}

The proof of our main result proceeds schematically as follows. We show that, when we consider a plumbing of cotangent bundles, there is a natural map of pointed spaces
    \begin{equation}
    \map_*(C,BGL_1\bbS)\to\mas^\theta,
    \end{equation}
where $\map_*(C,BGL_1\bbS)$ is the space of pointed maps $C\to BGL_1\bbS$ (here we fix a basepoint $c_0\in C$ and $GL_1\bbS$ is pointed at the identity), $\map_*(C,BGL_1\bbS)$ is pointed at the trivial map, and $\theta$ is a ``suitable choice of basepoint'' induced by the standard stable $\bbR$-polarization on $M$ and framed brane structure on $L_i$. The main insight into constructing this map is that there is a natural map $\scrP(L_0,L_1)\to\Omega\Sigma C$, where $\Omega\Sigma C$ is the free $\bbE_1$-group on $C$ via the James construction, and, since $BGL_1\bbS$ is an $\bbE_\infty$-group, we may use the universal property to extend any map $C\to BGL_1\bbS$ to $\Omega\Sigma C\to BGL_1\bbS$. We show that, under the hypotheses on $C$ in the statement of Theorem \ref{thm:main}, the $\infty$-functor 
    \begin{equation}\label{eq:restrictionintro}
    \map_*(C,BGL_1\bbS)_0\to \map_*(C,BGL_1\bbS)\to\mas^\theta\xrightarrow{\frakF^\sfr}\mathrm{Mod}_{\frakR^\sfr},
    \end{equation}
where $\map_*(C,BGL_1\bbS)_0$ is the connected component of the trivial map, cannot factor through a point (cf. Proposition \ref{prop:non-zero}). However, if the conclusion of Theorem \ref{thm:main} did not hold, then \eqref{eq:restrictionintro} would factor through a point and yield an immediate contradiction (cf. Subsection \ref{subsec:lagrangianintersections}). We do this essentially by identifying \eqref{eq:restrictionintro} with the $\infty$-functor which, on 0-simplices, sends a stable spherical fibration to its associated Thom spectrum smashed with $\frakR^\sfr$ (cf. Proposition \ref{prop:pullback}).

\begin{rem}
We strongly suspect Theorem \ref{thm:main} is true even in the non-orientable case. The reason we require orientability of $C$ is to ensure $H_*(C;\bbZ_{(p)})$ has a non-trivial $\bbZ_{(p)}$ summand for all sufficiently large primes $p$; this is an essential ingredient in our proof of Theorem \ref{thm:appendixb} which we use to prove our main result. We believe this issue can be overcome by working over Morava $K$-theories instead of $\bbS$ and suitably downgrading the definition of Maslov data by replacing $BGL_1\bbS$ with $BGL_1 K(n)$, where $K(n)$ is a suitably large Morava $K$-theory at the prime $p=2$. The advantage is that we now let $n$ be sufficiently large; in particular, we can still get 2-local information, i.e., the $\bbF_2$-summand in the top-degree of $H_*(C;\bbF_2)$ will ensure the existence of a non-trivial summand over $K(n)$. However, we digress, since it would require significant changes to an already well-established framework only to obtain a marginal improvement of our main result.
\end{rem}

In future work, we plan to investigate combining the truncated flow categories of \cite{Bon25} together with the various results of \cite{Bla,Bla24} and the present article to provide stronger lower bounds for (possibly degenerate) Lagrangian intersections in monotone symplectic manifolds.

\subsection*{Acknowledgments}
The first author would like to thank Johan Asplund and Noah Porcelli for helpful comments. The first author was partially supported by an NSF Graduate Research Fellowship award during this work. The second author would like to thank the sponsors William R. and Sara Hart Kimball for supporting him through the Stanford Graduate Fellowship at Stanford University, as well as the Simons Foundation for providing a research assistantship, as well as funding several conferences that benefited him in his research.

\section{Abouzaid-Blumberg's foundations for Floer homotopy}
In this section, we will review the definition of Abouzaid-Blumberg's \cite{AB24} stable $\infty$-category of flow categories (in particular, cf. Sections 3 and 4 of \emph{loc. cit.}).  Really, we will only review the definitions in a special case. We will not work with the full power of the theory since our applications to the exact setting do not require it, i.e., the compactified moduli spaces of Floer trajectories we will encounter are smooth manifolds with corners -- we do not require global Kuranishi charts. Moreover, for pedagogical reasons, we will: first define the objects (i.e., 0-simplices), then the morphisms (i.e., 1-simplices), and finally the higher simplices (i.e., $n$-simplices with $n\geq2$); even though it would be quicker to simply define all simplices at the same time.

\subsection{Objects}
Let $X$ be a smooth manifold with corners (by which we will always mean a $\langle k\rangle$-manifold in the sense of \cite{Lau00}). Associated to $X$ is a category $\calP_X$ whose objects are components of corner strata together with morphisms associated to reverse inclusion; in particular, the components of the interior of $X$ are minimal elements in $\calP_X$. There is a natural functor 
    \begin{equation}
    \codim:\calP_X\to\bbZ_{\geq0}
    \end{equation}
which sends a component of a corner stratum to its codimension in $X$ (note, we turn $\bbZ_{\geq0}$ into a category via its standard poset structure). 

Now, let $\calP$ be any category together with a functor $\codim:\calP\to\bbZ_{\geq0}$. For any object $p\in\ob\calP$, we denote by $\calP^p$ the overcategory of $p$.

\begin{defin}
We say $\calP$ is a \emph{model (for smooth manifolds with corners)} if, for any $p\in\ob\calP$, we have that $\calP^p$ is equivalent to the poset $2^{\{1,\ldots,\codim p\}}$; here, the poset structure on $2^{\{1,\ldots,\codim p\}}$ is given by inclusion, i.e., the initial object is the empty set.
\end{defin}

\begin{rem}
In practice, we will only work with models that are actually posets.
\end{rem}

Clearly, the example $\calP_X$ above is a model.

\begin{defin}
A \emph{stratified smooth manifold with corners} is a pair $(X,F:\calP_X\to\calP)$, where $X$ is a smooth manifold with corners and $F:\calP_X\to\calP$ is a functor from $\calP_X$ to a model $\calP$ such that (1) $F$ preserves the codimension and (2) for any $\sigma\in\ob\calP_X$, $F$ induces an equivalence $\calP^\sigma_X\xrightarrow{\sim}\calP^{F\sigma}$.
\end{defin}

\begin{defin}
A \emph{morphism} of stratified smooth manifolds with corners 
    \begin{equation}
    (X,F:\calP_X\to\calP)\to(X',F':\calP_{X'}\to\calP')
    \end{equation}
consists of a morphism of smooth manifolds with corners $X\to X'$ together with a functor $\calP\to\calP'$ making the following diagram commute: 
    \begin{equation}
    \begin{tikzcd}
    \calP_X\arrow[r]\arrow[d] & \calP_{X'}\arrow[d] \\
    \calP\arrow[r] & \calP'
    \end{tikzcd}
    .
    \end{equation}
\end{defin}

Given $(X,\calP_X\to\calP)$ and an object $p\in\ob\calP$, we may define the \emph{$p$-corner stratum} $(\partial^pX,\calP_{\partial^pX}\to\partial^p\calP)$ via defining $\partial^pX$ to consist of pairs $(p\to q,x)$, where $p\to q$ is an object of the undercategory $\partial^p\calP$ of $p$ and $x$ is an element of $X$ lying in a stratum labeled by $q$. Note, $\partial^pX$ is a stratified submanifold with corners of $X$.

Now, let $\calP$ be any set. For any two $x,y\in\calP$, we may construct a model $\calP(x,y)$ as follows. The objects are trees with only bivalent vertices whose edges are labeled by elements of $\calP$ such that the incoming resp. outgoing (semi-infinite) leaf is labeled by $x$ resp. $y$.\footnote{We can imagine an object of $\calP(x,y)$ as a tuple $(x,z,\ldots,z',y)$.} The morphisms are given by collapsing internal edges; in particular, the initial object is the tree with only the incoming and outgoing leaves. The functor $\codim:\calP(x,y)\to\bbZ_{\geq0}$ is given by counting the number of internal edges. For any three $x,y,z\in\calP$, there is a natural functor 
    \begin{equation}
    \calP(x,z)\times\calP(z,y)\to\calP(x,y)
    \end{equation}
given by concatenation. In particular, the collection 
    \begin{equation}
    \big\{\calP(x,y)\big\}_{x,y\in\calP}
    \end{equation}
yield a strict 2-category which, by an abuse of notation, we denote by $\calP$.

\begin{defin}
An \emph{(unstructured) flow category} $\bbX$ consists of the following data. 
\begin{enumerate}
\item An object set $\calP$.
\item For any two $x,y\in\calP$, a stratified compact smooth manifold with corners 
    \begin{equation}
    \big(\bbX(x,y),\calP(x,y)\big).
    \end{equation}
\item For any three $x,y,z\in\calP$, a morphism of compact stratified smooth manifolds with corners
    \begin{equation}
    \bbX(x,z)\times\bbX(z,y)\to\bbX(x,y),
    \end{equation}
which is an embedding of a codimension 1 boundary stratum that lifts the functor 
    \begin{equation}
    \calP(x,z)\times\calP(z,y)\to\calP(x,y)
    \end{equation}
(and, moreover, the codimension 1 boundary strata of $\bbX(x,y)$ are enumerated by such morphisms), such that the natural associativity diagram associated to four elements commutes. 
\end{enumerate}
\end{defin}

A quick example of a flow category is the \emph{unit} $\unit$ which consists of a single object with no morphisms. A more robust example is the following. Let $X$ be a closed smooth manifold together with a Morse-Smale pair $(f,g)$; there is a flow category $\bbM^{f,g}_X$ whose objects are the critical points of $f$ and whose morphism spaces are the compactified moduli spaces of Morse trajectories connecting critical points.

\begin{defin}\label{defin:framedflowcategory}
A \emph{spherically framed flow category} is a flow category $\bbX$ together with the following data. 
\begin{enumerate}
\item A real virtual vector space $V_x$ for any $x\in\ob\bbX$.
\item A real vector bundle $W(x,y)\to\bbX(x,y)$ for any two $x,y\in\ob\bbX$.
\item A real virtual bundle $I(x,y)\to\bbX(x,y)$, together with a trivialization as a stable spherical fibration, for any two $x,y\in\ob\bbX$. 
\item An isomorphism of real virtual bundles
    \begin{equation}
    T\bbX(x,y)+\underline{\bbR}+V_y+W(x,y)\cong V_x+I(x,y)+W(x,y),
    \end{equation}
for any two $x,y\in\ob\bbX$, such that the natural associativity diagram associated to three elements commutes.
\end{enumerate}
\end{defin}

\begin{rem}
We may define a \emph{framed flow category} analogously to Definition \ref{defin:framedflowcategory}, except we require $I(x,y)$ to be trivialized as a real virtual bundle.
\end{rem}

We denote by $\flow^\sfr$ resp. $\flow^\fr$ the stable $\infty$-category of spherically framed resp. framed flow categories (we will define the higher simplices shortly).

\begin{rem}\label{rem:structuredflow}
\cite[Theorem 1.6]{AB24} shows that there is a stable $\infty$-category $\flow^S$, where $S$ is some tangential structure (i.e., spherically framed or framed), whose 0-simplices are ``tangentially structured'' flow categories. Moreover, we define the spectrum $\frakX$ associated to a tangentially structured flow category $\bbX$ as the mapping spectrum $\flow^S(\unit,\bbX)$. It is in fact known that $\flow^S(\unit,\unit)$ is a ring spectrum whose homotopy groups are the bordism groups of the relevant bordism theory specified by $S$, and moreover, that $\frakX$ is a module over this ring spectrum. For example, there is an equivalence of stable infinity categories $\flow^\fr\cong\spectra$, where $\spectra$ is the stable $\infty$-category of spectra, cf. Proposition 1.10 of \emph{loc. cit.} (note, $\frakR^\fr\equiv\flow^\fr(\unit,\unit)\cong\bbS$ is the spectrum associated to the bordism theory $\Omega^\fr_*(\cdot)$ of stably framed closed smooth manifolds). We define $\frakR^\sfr\equiv\flow^\sfr(\unit,\unit)$; note, $\pi_*\frakR^\sfr$ is the bordism ring $\Omega^\sfr_*(*)$ associated to the bordism theory of closed smooth manifolds whose tangent bundles come with a choice of trivialization as a stable spherical fibration.
\end{rem}

Again, a quick example of a framed flow category is the \emph{shifted unit} $\Sigma^d\unit$, $d\in\bbZ$, which consists of a single object with no morphisms together with the framing given by a real virtual vector space of virtual rank $d$. And again, a more robust example is endowing $\bbM^{f,g}_X$ with its ``standard'' framing; the associated spectrum of $\bbM^{f,g}_X$ is then the suspension spectrum $\Sigma^\infty_+X$, cf. \cite[Subsection 4.3]{Bla24}. In fact, if $E:X\to BGL_1\bbS$ is any stable spherical fibration, then, by a straightforward extension of \emph{loc. cit.}, we may twist the standard framing to get a new framed flow category $\bbM^{f,g}_{X,E}$; the associated spectrum of $\bbM^{f,g}_{X,E}$ is then the Thom spectrum $X^E$. The analogous statement for twisting by a real virtual bundle $X\to BO$ has already been proven in the Cohen-Jones-Segal framework, cf. \cite[Theorem 5.9]{Bon24} and \cite[Theorem 3.11]{CK23}. Finally, we may consider $\bbM^{f,g}_{X,E}$ as a spherically framed flow category, and we see its associated spectrum is $X^E\wedge \frakR^\sfr$, cf. Proposition \ref{prop:sfrfr}.

\subsection{Morphisms}
Now, let $\vec{\calP}\equiv(\calP_0,\calP_1)$ be any 2-tuple of sets. For any two $x\in\calP_j$ and $y\in\calP_{j'}$, $0\leq j\leq j'\leq1$, we may construct a model $\vec{\calP}(x,y)$ as follows. The objects are trees with only bivalent vertices whose edges are labeled by elements of $\calP_k$, $j\leq k\leq j'$, such that elements on edges appear in ascending order from left to right and the incoming resp. outgoing leaf is labeled by $x$ resp. $y$. The morphisms are given by collapsing internal edges. The functor $\codim:\vec{\calP}(x,y)\to\bbZ_{\geq0}$ is given by counting the number of internal edges. For any three $x,y,z\in\calP$, there is a natural functor 
    \begin{equation}
    \vec{\calP}(x,z)\times\vec{\calP}(z,y)\to\vec{\calP}(x,y)
    \end{equation}
given by concatenation. In particular, the collection 
    \begin{equation}
    \big\{\vec{\calP}(x,y)\big\}_{x,y\in\calP}
    \end{equation}
yields a strict 2-category which, by an abuse of notation, we denote by $\vec{\calP}$. There is a natural inclusion of strict 2-categories 
    \begin{equation}
    \calP_j\to\vec{\calP}.
    \end{equation}

\begin{defin}
An \emph{(unstructured) flow bimodule} $\bbX_{01}:\bbX_0\to\bbX_1$ consists of the following data. 
\begin{enumerate}
\item An object 2-tuple $\vec{\calP}\equiv(\calP_0,\calP_1)$.
\item A flow category $\bbX_j$ with object set $\calP_j$.
\item For any two $x\in\calP_0$ and $y\in\calP_1$, a stratified compact smooth manifold with corners 
    \begin{equation}
    \big(\bbX_{01}(x,y),\vec{\calP}(x,y)\big).
    \end{equation}
\item For any two $x,z\in\calP_0$ and $y\in\calP_1$, a morphism of compact stratified smooth manifolds with corners
    \begin{equation}
    \bbX_0(x,z)\times\bbX_{01}(z,y)\to\bbX_{01}(x,y),
    \end{equation}
which is an embedding of a codimension 1 boundary stratum that lifts the functor 
    \begin{equation}
    \vec{\calP}(x,z)\times\vec{\calP}(z,y)\to\vec{\calP}(x,y)
    \end{equation}
such that the natural associativity diagram associated to four elements commutes. 
\item The analogous condition for any $x\in\calP_0$ and two $z,y\in\calP_1$.
\item For any $x\in\calP_0$ and $y\in\calP_1$, the codimension 1 boundary strata of $\bbX_{01}(x,y)$ are enumerated by morphisms of the above form; moreover, the natural associativity diagram intertwining these commutes.
\end{enumerate}
\end{defin}

Consider two flow bimodules $\bbX_{01}$ and $\bbX_{12}$; we may ``compose'' them as follows. Essentially, we define a flow bimodule $\bbX_{02}\equiv\bbX_{01}\circ\bbX_{12}$ such that, for any $x\in\calP_0$ and $y\in\calP_2$, 
    \begin{equation}
    \bbX_{02}(x,y)\equiv\coprod_{z\in\calP_1}\bbX_{01}(x,z)\times\bbX_{12}(z,y)/\sim, 
    \end{equation}
where the equivalence relation identifies the images 
    \begin{equation}
    \bbX_{01}(x,z)\times\bbX_{12}(z,y)\leftarrow\bbX_{01}(x,z)\times\bbX_1(z,z')\times\bbX_{12}(z',y)\rightarrow\bbX_{01}(x,z')\times\bbX_{12}(z',y);
    \end{equation}
of course, we need to smooth the various $\bbX_{02}(x,y)$, cf. \cite[Section 5]{AB24}, also \cite[\S 6]{Bon24}. Of course, it is more in line with the $\infty$-categorical philosophy to think of composition as being only well-defined up to contractible choice, using 2-simplices to be defined shortly; but in this case we happen to have a ``preferred'' such choice.

\begin{defin}\label{defin:framedflowbimodule}
A \emph{spherically framed flow bimodule} is a flow bimodule $\bbX_{01}:\bbX_0\to\bbX_1$ together with the following data. 
\begin{enumerate}
\item A framing on $\bbX_j$.
\item A real virtual bundle $W(x,y)\to\bbX_{01}(x,y)$ for any two $x\in\ob\bbX_0$ and $y\in\ob\bbX_1$.
\item A real virtual bundle $I(x,y)\to\bbX_{01}(x,y)$, together with a trivialization as a stable spherical fibration, for any two $x\in\ob\bbX_0$ and $y\in\ob\bbX_1$.
\item An isomorphism of real virtual bundles 
    \begin{equation}
    T\bbX_{01}(x,y)+V_y+W(x,y)\cong V_x+I(x,y)+W(x,y),
    \end{equation}
for any two $x\in\ob\bbX_0$ and $y\in\ob\bbX_1$, such that the natural associativity diagrams associated to three elements commute.
\end{enumerate}
\end{defin}

\begin{rem}
We may define a \emph{framed flow bimodule} analogously to Definition \ref{defin:framedflowbimodule}, except we require $I(x,y)$ to be trivial as a real virtual bundle.
\end{rem}

\subsection{Higher simplices}
Finally, let $\vec{\calP}\equiv(\calP_0,\ldots,\calP_n)$ be any $(n+1)$-tuple of sets, $n\geq2$. For any two $x\in\calP_j$ and $y\in\calP_{j'}$, $0\leq j\leq j'\leq n$, we may construct a model $\vec{\calP}(x,y)$ as follows. The objects are trees with only bivalent vertices whose:
\begin{itemize}
\item edges are labeled by elements of $\calP_k$, $j\leq k\leq j'$, such that elements on edges appear in ascending order from left to right and the incoming, resp. outgoing leaf is labeled by $x$, resp. $y$; 
\item and vertices are labeled by a subset of $\{k+1,\ldots,k'-1\}$, where the edges to the left, resp. right of a vertex are labeled by an element of $\calP_k$, resp. $\calP_{k'}$.
\end{itemize}
The morphisms are given by collapsing internal edges; the labels on vertices behave as follows. Suppose $T\to T'$ is a morphism, the label of each vertex $v$ of $T$ contains the union of the labels of the vertices $v'$ of $T'$ which are collapsed to it together with any $k$, $j<k<j'$, with the property that all edges of $T'$ labeled by an element of $\calP_k$ are collapsed to $v$. The functor $\codim:\vec{\calP}(x,y)\to\bbZ_{\geq0}$ is given by counting (1) the number of internal edges and (2) the number of $k$, $j<k<j'$, with the property that $k$ could appear as a label of a vertex but does not (i.e., if a vertex is labeled by a subset $S\subset\{k+1,\ldots,k'-1\}$, then the codimension of this tree is increased by the cardinality of the complement $S^c$). For any three $x,y,z\in\calP$, there is a natural functor 
    \begin{equation}
    \vec{\calP}(x,z)\times\vec{\calP}(z,y)\to\vec{\calP}(x,y)
    \end{equation}
given by concatenation. In particular, the collection 
    \begin{equation}
    \big\{\vec{\calP}(x,y)\big\}_{x,y\in\calP}
    \end{equation}
yields a strict 2-category which, by an abuse of notation, we denote by $\vec{\calP}$. There is a natural inclusion of strict 2-categories 
    \begin{equation}
    \partial^j\vec{\calP}\equiv\calP_j\to\vec{\calP}.
    \end{equation} 
It is straightforward to see that, for any $0\leq j<j'\leq n$, there is an equality of strict 2-categories $\partial^{j'}\partial^j\vec{\calP}=\partial^{j-1}\partial^{j'}\vec{\calP}$ which is compatible with the inclusions.

\begin{defin}
A(n elementary\footnote{In the general theory, a flow simplex is built out of elementary flow simplices. This is essentially due to the fact that one considers stratified derived orbifolds with corners whose codimension 1 boundary strata fall into two ``types'' -- ones which are enumerated by ``strong equivalences'' and ones which are enumerated by actual equalities. A Floer-theoretic interpretation for this is that the global Kuranishi charts for Floer-type moduli spaces involved in, say, a homotopy of continuation maps has codimension 1 boundary strata induced by gluing and codimension 1 boundary strata induced by breaking of families of continuation data; the gluing requires stabilizing the obstruction bundles while the breaking of families of continuation data does not. Since we are working with stratified compact smooth manifolds with corners, there are no obstruction bundles, and we may safely ignore this point.}) flow $n$-simplex $\bbX\equiv\bbX_{0\cdots n}$ consists of the following data. 
\begin{enumerate}
\item An object $(n+1)$-tuple $\vec{\calP}\equiv(\calP_0,\ldots,\calP_n)$.
\item For any two $x\in\calP_j$ and $y\in\calP_{j'}$, $0\leq j\leq j'\leq n$, a stratified compact smooth manifold with corners 
    \begin{equation}
    \big(\bbX(x,y),\vec{\calP}(x,y)\big).
    \end{equation}
\item For any three $x\in\calP_j$, $y\in\calP_{j'}$, and $z\in\calP_k$, $0\leq j\leq k\leq j'\leq n$, a morphism of stratified compact smooth manifolds with corners 
    \begin{equation}
    \bbX(x,z)\times\bbX(z,x)\to\bbX(x,y)
    \end{equation}
which is an embedding of a codimension 1 boundary stratum that lifts the functor 
    \begin{equation}
    \vec{\calP}(x,z)\times\vec{\calP}(z,y)\to\vec{\calP}(x,y)
    \end{equation}
such that the natural associativity diagram associated to four elements commutes. 
\item For any two $x\in\calP_j$ and $y\in\calP_{j'}$, $0\leq j\leq j'\leq n$, the codimension 1 boundary strata of $\bbX(x,y)$ are enumerated by (1) morphisms of the above form and (2) the stratified compact submanifolds with corners $\partial^k\bbX(x,y)$, $j<k<j'$; moreover, the natural associativity diagrams intertwining these commutes.
\end{enumerate}
\end{defin}

\begin{defin}\label{defin:framedflowsimplex}
A(n \emph{elementary) spherically framed flow $n$-simplex} is a flow $n$-simplex together with the following data.
\begin{enumerate}
\item A real virtual vector space $V_x$ for any $x\in\calP_j$.
\item A real vector bundle $W(x,y)\to\bbX(x,y)$ for any two $x\in\calP_j$ and $y\in\calP_{j'}$, $0\leq j\leq j'\leq n$.
\item A real virtual bundle $I(x,y)\to\bbX(x,y)$, together with a trivialization as a stable spherical fibration, for any two $x\in\calP_j$ and $y\in\calP_{j'}$, $0\leq j\leq j'\leq n$. 
\item An equivalence 
    \begin{equation}
    T\bbX(x,y)+\underline{\bbR}+V_y+W(x,y)\cong V_x+\underline{\bbR}^{\abs{\{j+1,\ldots,j'\}}}+I(x,y)+W(x,y)
    \end{equation}
for any two $x\in\calP_j$ and $y\in\calP_{j'}$, $0\leq j\leq j'\leq n$, such that the natural associativity diagram associated to three elements commutes.
\end{enumerate}
\end{defin}

\begin{rem}
We may define a(n \emph{elementary) framed flow $n$-simplex} analogously to Definition \ref{defin:framedflowsimplex}, except we require $I(x,y)$ to be trivial as a real virtual bundle.
\end{rem}

\section{Unparameterized Lagrangian Floer theory}
\subsection{Admissible Floer data}
Let $(M,\omega=d\lambda)$ be a Liouville manifold. Recall, this means $M$ is the completion of a Liouville domain $\widehat{M}$ given by gluing on the positive-half of the symplectization: 
    \begin{equation}
    M=\widehat{M}\cup_{\partial\widehat{M}}\Big([1,+\infty)_r\times\partial\widehat{M}\Big).
    \end{equation}
Note, we denote the radial coordinate on the cylindrical end by $r$. Let $H\in C^\infty(I\times M)$, $I\equiv[0,1]$, be a (time-dependent) Hamiltonian; we define $H_t\equiv H(t,\cdot)$. We will denote by: $X_{H_t}$ the Hamiltonian vector field,     
    \begin{equation}
    \omega\big(X_{H_t},\cdot\big)=-dH_t;
    \end{equation}
$\phi^t_{H_t}$ the Hamiltonian flow, 
    \begin{equation}
    \partial_t\phi^t_{H_t}=X_{H_t}\circ\phi^t_{H_t},\;\;\phi^0_{H_t}=\identity_M;
    \end{equation}
and $\chi(L_0,L_1;H)$ the set of time-1 Hamiltonian chords connecting $L_0$ to $L_1$, 
    \begin{equation}
    \chi(L_0,L_1;H)\equiv\big\{x\in C^\infty(I,M):\dot{x}(t)=X_{H_t}\big(x(t)\big),x(i)\in L_i\big\}.
    \end{equation}
Recall, $H_t$ is called non-degenerate if $\phi^1_{H_t}(L_0)$ intersects $L_1$ transversely; this is a generic condition. Moreover, we will assume $H_t$ is linear at infinity of slope $\tau\in\bbR$, i.e., 
    \begin{equation}
    H_t(r,p)=\tau\cdot r,\;\;r\gg1.
    \end{equation}
(Note, we choose $\tau$ such that it is not the length of a Reeb chord connecting $\partial L_0$ to $\partial L_1$.) We call Hamiltonians satisfying the two previous conditions \emph{admissible (of slope $\tau$)}. We now fix $\tau$ and denote the space of admissible Hamiltonians (of slope $\tau$) by $\calH$. The following result is standard.

\begin{lem}
$\calH$ is non-empty and contractible. 
\end{lem}

Now, let $J\equiv\{J_t\}_{t\in[0,1]}$ be a (time-dependent $\omega$-compatible) almost complex structure. We will assume $J$ is of contact type at infinity, i.e., each $J_t$ (1) preserves the contact-type distribution $\ker\big(\lambda\vert_{\partial\widehat{M}}\big)$ and (2) takes the Liouville vector field of $(M,\omega)$ to the Reeb vector field of $\big(\partial\widehat{M},\lambda\vert_{\partial\widehat{M}}\big)$. We call these almost complex structures \emph{admissible}, and we denote the space of such almost complex structures by $\calJ$. The following result is standard.

\begin{lem}
$\calJ$ is non-empty and contractible.
\end{lem}

\subsection{Lagrangian Floer homotopy type}
We now assume we are working under Assumption \ref{assu:main}. Let $(H,J)\in\calH\times\calJ$. Given any two $x,y\in\chi(L_0,L_1;H)$, we consider the moduli space $\widetilde{\scrF}^{H,J}(x,y)$ of Floer trajectories connecting $x$ to $y$, i.e., strips $u:Z\equiv\bbR_s\times I_t\to M$ satisfying 
    \begin{equation}
    \begin{cases}
    \partial_su+J_t\big(\partial_tu-X_{H_t}(u)\big)=0 \\
    u(s,i)\in L_i \\
    \lim_{s\to-\infty}u(s,t)=x(t) \\
    \lim_{s\to+\infty}u(s,t)=y(t)
    \end{cases}
    .
    \end{equation}
Recall, $\widetilde{\scrF}^{H,J}(x,y)$ is generically a smooth manifold of dimension $\mu(x)-\mu(y)$, where $\mu(\cdot)$ is the Maslov index.\footnote{Recall, Assumption \ref{assu:main} implies, in particular, that $L_i$ is graded in the sense of Seidel \cite{Sei00}.} This occurs when each member of the family of linearized operators (with totally real boundary conditions) 
    \begin{equation}
    D\overline{\partial}_{H,J}\equiv\big\{D(\overline{\partial}_{H,J})_u:W^{1,2}(Z;u^*TM,u^*TL_i)\to L^2(Z,u^*TM)\big\},
    \end{equation}
given by linearizing the Floer equation, is a surjective Fredholm operator. In particular, $T\widetilde{\scrF}^{H,J}(x,y)$ is classified by the index bundle $\ind D\overline{\partial}_{H,J}$. We call pairs $(H,J)$ with this property \emph{regular}.

When $x\neq y$, there is a smooth proper $\bbR$-action on $\widetilde{\scrF}^{H,J}(x,y)$ given by time-shift in the $s$-coordinate, and we denote the $\bbR$-quotient by $\scrF^{H,J}(x,y)$. Moreover, there is a natural Gromov-compactification $\bbF^{H,J}(x,y)$ of $\scrF^{H,J}(x,y)$ given by allowing breakings at time-1 Hamiltonian chords. It was shown in \cite[Section 6]{Lar21} that $\bbF^{H,J}(x,y)$ is a stratified compact smooth manifold with corners whose stratification is given by breakings at time-1 Hamiltonian chords. Moreover, it was shown in Section 7 of \emph{loc. cit.} how a stable framing of the background symplectic manifold, together with a null-homotopy of the stable Gauss map of each Lagrangian, implies there are coherent stable framings on the compactified moduli spaces of Floer trajectories. This has also been proven when the background symplectic manifold admits a stable $\bbR$-polarization and each Lagrangian has a framed brane structure, cf. \cite{ADP24,Bla24,PS24a,PS24b}. We will now prove the following folklore statement. 

\begin{prop}\label{prop:unparamterizedframing}
Suppose Assumption \ref{assu:main}, then the compactified moduli spaces of Floer trajectories admit coherent $\bbS$-orientations.
\end{prop}

\begin{proof}
Given any two $x,y\in\chi(L_0,L_1;H)$, $T\widetilde{\scrF}^{H,J}(x,y)$ is classified by a map
    \begin{equation}
    D\overline{\partial}_{H,J}:\widetilde{\scrF}^{H,J}(x,y)\to \bbZ\times BO
    \end{equation}
which takes a Floer trajectory $u$ connecting $x$ to $y$ to $\ind D(\overline{\partial}_{H,J})_u$. It was shown in \cite[Section 7]{Lar21} how $T\bbF^{H,J}(x,y)$, coming from the stratified compact smooth manifold with corner structure induced by gluing, is related to (an extension to the Gromov-compactification of) the smooth vector bundle $\ind D\overline{\partial}_{H,J}$ (also, cf. \cite[Sections 6.4, 6.5]{PS24b}). We have the basic relation 
    \begin{equation}\label{eq:isomorphisms}
    T\scrF^{H,J}(x,y)\oplus\underline{\bbR}\cong\ind D\overline{\partial}_{H,J}.
    \end{equation}
In particular, the compatibility of: (1) the natural decomposition of $T\bbF^{H,J}(x,y)$ over a stratum into the direct sum of the tangent bundles of that stratum together with the collar directions, (2) the natural decomposition of $\ind D\overline{\partial}_{H,J}$ over a stratum into the direct sum of the index bundles of the families of linearized operators associated to that stratum, and (3) the various isomorphism of the form \eqref{eq:isomorphisms}; was established. The upshot is that $T\bbF^{H,J}(x,y)\oplus\underline{\bbR}$ is classified by a map 
    \begin{equation}
    D\overline{\partial}_{H,J}:\bbF^{H,J}(x,y)\to\bbZ\times BO
    \end{equation}
related to the index bundles of the various families of linearized operators which is manifestly compatible with gluing. Consider $[(u_0,\ldots,u_r)]\in\bbF^{H,J}(x,y)$. We have the asymptotic operator $D(\overline{\partial}_{H,J})_x$ resp. $D(\overline{\partial}_{H,J})_y$ of $D(\overline{\partial}_{H,J})_{u_0}$ resp. $D(\overline{\partial}_{H,J})_{u_r}$  at $-\infty$ resp. $+\infty$. Note, these asymptotic operators are invertible operators
    \begin{align}
    W^{1,2}(I;x^*TM,x^*TL_i)&\to L^2(I;x^*TM), \\
    W^{1,2}(I;y^*TM,y^*TL_i)&\to L^2(I;y^*TM),
    \end{align}
respectively, which are independent of $[(u_0,\ldots,u_r)]$. As described in \cite[Subsection 5.4]{Bla24}, we may construct a \emph{Floer abstract cap} for $x$, i.e., a canonical index $\mu(x)$ Fredholm operator
    \begin{equation}
    T_{\frakF,x}:W^{1,2}(Z;x^*TM,x^*TL_i)\to L^2(Z;x^*TM)
    \end{equation}
with asymptotic operator $D(\overline{\partial}_{H,J})_x$ resp. $\partial_t+\kappa\cdot\identity$ at $-\infty$ resp. $+\infty$, where $\kappa\in\bbR_{>0}$ is sufficiently small. Note, we have a canonical isomorphism $\ind T_{\frakF,x}\cong\bbR^{\mu(x)}$. In particular, we may consider the map
    \begin{equation}
    D\overline{\partial}_\mathrm{aux}:\bbF^{H,J}(x,y)\to BO
    \end{equation}
given by 
    \begin{equation}
    D\overline{\partial}_{H,J}+T_{\frakF,y}-T_{\frakF,x},
    \end{equation}
where $T_{\frakF,x}$ resp. $T_{\frakF,y}$ is the obvious constant map $\bbF^{H,J}(x,y)\to \bbZ\times BO$; note, this map classifies a virtual rank 0 real virtual bundle. We have the following composition:
    \begin{equation}\label{eq:coherentsaux1}
    \bbF^{H,J}(x,y)\xrightarrow{D\overline{\partial}_\mathrm{aux}} BO\xrightarrow{BJ}BGL_1\bbS.
    \end{equation}
On the other hand, we have an obvious map 
    \begin{equation}
    \bbF^{H,J}(x,y)\to\Omega_{x,y}\scrP(L_0,L_1)\simeq\Omega\scrP(L_0,L_1),
    \end{equation}
where $\Omega_{x,y}\scrP(L_0,L_1)$ denotes the space of paths in $\scrP(L_0,L_1)$ connecting $x$ to $y$, which is well-defined up to homotopy. We may consider the composition
    \begin{equation}\label{eq:coherentsaux2}
    \bbF^{H,J}(x,y)\to\Omega\scrP(L_0,L_1)\to\Omega(U/O)\xrightarrow{\sim}\Omega^\infty KO\to BO\xrightarrow{BJ}BGL_1\bbS,
    \end{equation}
where the composition of the second, third, fourth, and fifth map is the looping of \eqref{eq:assu}. By \cite[\S 2.3]{Bon25}, \eqref{eq:coherentsaux2} is equivalent to \eqref{eq:coherentsaux1}. Finally, by Assumption \ref{assu:main}, we have a choice of null-homotopy of \eqref{eq:coherentsaux2} induced by $\theta$; this is equivalent to saying we have an induced choice of virtual rank 0 real virtual bundle $I_\theta(x,y)\to\bbF^{H,J}(x,y)$, together with a trivialization as a stable spherical fibration, and a choice of isomorphism of real virtual bundles
    \begin{equation}
    T\bbF^{H,J}(x,y)+\underline{\bbR}+\underline{\bbR}^{\mu(y)}\cong\underline{\bbR}^{\mu(x)}+I_\theta(x,y).
    \end{equation}
Moreover, the natural associativity diagram associated to three elements commutes (i.e., we have compatibility with gluing) by (1) canonical isomorphisms of the form 
    \begin{equation}
    \ind T_{\frakF,x}-\ind T_{\frakF,x}\cong0,
    \end{equation}
(2) the fact that (the extension to the Gromov-compactification of) $\ind D\overline{\partial}_{H,J}$ is compatible with gluing, and (3) the fact that the Bott isomorphism intertwines the $\bbE_\infty$-structure given by addition of index bundles and the $\bbE_1$-structure given by loop concatenation via the Eckmann-Hilton argument.
\end{proof}

\begin{defin}
The \emph{(spherically framed) Lagrangian Floer homotopy type} of $(L_0,L_1)$, associated to a choice of regular admissible pair $(H,J)$ and null-homotopy $\theta$ of \eqref{eq:assu}, is the $\frakR^\sfr$-module spectrum $\frakF^{H,J,\theta}$ associated to the spherically framed flow category $\bbF^{H,J,\theta}$ whose objects are time-1 Hamiltonian chords of $H$ connecting $L_0$ to $L_1$, morphism spaces are the compactified moduli spaces of Floer trajectories connecting time-1 Hamiltonian chords, and spherical framing is induced by Proposition \ref{prop:unparamterizedframing}.
\end{defin}

\section{Parameterized Lagrangian Floer theory}
\subsection{The parameter space of Maslov data}
Recall, we denote by $\mas$ the space of Maslov data and by $\mas^\theta$ the pointed space $\mas,\theta\big)$. The following lemma is standard algebraic topology.

\begin{lem}\label{lem:hofib}
A choice of $\theta$ is equivalent to a choice of (1) a lift $F_\theta$ to the homotopy fiber of $B^2J$ and (2) a homotopy witnessing homotopy commutativity of the diagram
    \begin{equation}
    \begin{tikzcd}
    & \mathrm{hofib}B^2J\arrow[d] \\
    \scrP(L_0,L_1)\arrow[r]\arrow[ur,"F_\theta"] & B^2O
    \end{tikzcd}
    .
    \end{equation}
\end{lem}

We are particularly interested in the case of plumbings of cotangent bundles; we recall the construction now.

\begin{example}\label{example:plumbing}
Let $Q_i$ be two smooth manifolds together with an embedding of a closed smooth manifold $C\to Q_i$ such that the normal bundle $N_CQ_i\to C$ is isomorphic to a fixed real vector bundle $NC\to C$, where we think of $N_CQ_i$ as a tubular neighborhood of $C$ in $Q_i$ via picking a suitable Riemannian metric. By Weinstein's theorem, there exists a symplectomorphism $\chi:U\to V$,
where $U\subset TQ_i$ is a neighborhood of $C$ and $V\subset\tau^*(NC\otimes_\bbR\underline{\bbC})$ is a neighborhood of the 0-section (here, $\tau:T^*C\to C$ is the projection). We may choose $\chi$ so that 
    \begin{equation}
    \chi(U\cap Q_i)=N_CQ_i.
    \end{equation}
By construction, there are two natural totally real subbundles of $NC\otimes_\bbR\underline{\bbC}$; namely, $NC$ and $\sqrt{-1}NC$. We may build a new Liouville manifold 
    \begin{equation}
    \big(M\equiv T^*Q_0\cup_{NC}T^*Q_1,\omega=d\lambda\big)
    \end{equation}
by gluing $T^*Q_0$ to $T^*Q_1$ along $NC$ after complex multiplication. Note, the Liouville form $\lambda$ agrees with the standard Liouville form on $T^*Q_i$ away from $C$ and agrees with the direct sum of the standard Liouville forms on $T^*C$ and $NC\otimes_\bbR\underline{\bbC}$ near $C$. Moreover, there exist two Lagrangians $L_i\subset M$, where $L_i$ is diffeomorphic to $Q_i$, such that $(L_0,L_1)$ intersects cleanly along $C$.
\end{example}

While $\mas^\theta$ may be quite complicated in general, in certain geometric situations it may admit maps from more manageable spaces. Let $c_0\in C$ be a fixed choice of basepoint; recall, we denote by $\map_*\big(C,BGL_1\bbS\big)$ the space of pointed maps $C\to BGL_1\bbS$ ($GL_1\bbS$ is pointed at the identity), i.e., this is the space of stable spherical fibrations $\Phi:C\to BGL_1\bbS$ which have a fixed choice of trivialization $\Phi\vert_{c_0}\simeq\bbS$. Consider the following result.

\begin{lem}\label{lem:BGL}
Suppose $M$ is the plumbing of two cotangent bundles $T^*Q_i$ along $NC$, then we have a natural map of pointed spaces
    \begin{equation}
    \map_*(C,BGL_1\bbS)\to\mas^\theta,
    \end{equation}
where $\map_*(C,BGL_1\bbS)$ is pointed at the trivial map.
\end{lem}

\begin{proof}
First, let $C(L_i)$ be the reduced cone over the inclusion $L_i\to M$. There is a natural pointed map (well-defined up to homotopy) 
    \begin{equation}\label{eq:maptoloops}
    \scrP(L_0,L_1)\to\Omega\Big(M\cup\big(C(L_0)\vee C(L_1)\big)\Big),
    \end{equation}
where $\scrP(L_0,L_1)$ is pointed at the constant path at $c_0$ and $\Omega\Big(M\cup\big(C(L_0)\vee C(L_1)\big)\Big)$ is pointed at the constant loop at $c_0$, defined as follows. We take a path connecting $L_0$ to $L_1$ and connect the starting resp. ending point of this path to $c_0$ via a path in $C(L_0)$ resp. $C(L_1)$. In fact, since $M$ retracts onto its isotropic skeleton $L_0\cup_C L_1$, it is straightforward to see a homotopy equivalence 
    \begin{equation}\label{eq:sigmac}
    M\cup\big(C(L_0)\vee C(L_1)\big)\simeq\Sigma C.
    \end{equation}
Summarizing \eqref{eq:maptoloops} and \eqref{eq:sigmac}, we get a map
    \begin{equation}\label{eq:maptofree}
        \scrP(L_0,L_1) \to \Omega\Sigma C
    \end{equation}
from the pathspace to the free $\mathbb E_1$-group on $C$ (i.e., the codomain is the James construction). Since $BGL_1(\mathbb S)$ is itself an $\mathbb E_1$-group (in fact, it is an $\mathbb E_\infty$-group), this induces a map
    \begin{equation}
    \map_*(C, BGL_1\bbS)\to\map_*(\Omega\Sigma C, BGL_1\mathbb S)\to\map_*(\scrP(L_0,L_1),BGL_1\mathbb S),
    \end{equation}
where the first map is induced by the James construction. Alternatively, we may describe this composite by reinterpreting the $\mathbb E_1$-group structure as being a 1-fold loopspace, and then using (1) suspension-loop adjunction, (2) looping, and (3) precomposition:
    \begin{equation}
    \begin{tikzcd}
    \map_*(C,BGL_1\bbS)\cong\map_*(\Sigma C,B^2GL_1\bbS)\arrow[d] \\
    \map_*(\Omega\Sigma C,\Omega B^2GL_1\bbS)\cong\map_*(\Omega\Sigma C,BGL_1\bbS)\arrow[d] \\
    \map_*\big(\scrP(L_0,L_1),BGL_1\bbS\big)
    \end{tikzcd}
    .
    \end{equation}
Thus, it remains to construct a map
    \begin{equation}\label{eq:BGLaux}
    \map_*\big(\scrP(L_0,L_1),BGL_1\bbS\big)\to\mas^\theta,
    \end{equation}
and we proceed as follows. Given $h:\scrP(L_0,L_1)\to BGL_1\bbS$, we denote by $\widetilde{h}$ the composition 
    \begin{equation}
    \scrP(L_0,L_1)\xrightarrow{h}BGL_1\bbS\to\mathrm{hofib}B^2J.
    \end{equation}
Note, $h$ itself determines a null-homotopy of the composition 
    \begin{equation}
    \scrP(L_0,L_1)\xrightarrow{\widetilde{h}}\mathrm{hofib}B^2J\to B^2O.
    \end{equation}
Consider $F_\theta$ together with its homotopy witnessing homotopy commutativity of the diagram
    \begin{equation}
    \begin{tikzcd}
    & \mathrm{hofib}B^2J\arrow[d] \\
    \scrP(L_0,L_1)\arrow[r]\arrow[ur,"F_\theta"] & B^2O
    \end{tikzcd}
    .
    \end{equation}
We may define \eqref{eq:BGLaux} by sending $h$ to $F_\theta+\widetilde{h}$ together with the obvious homotopy witnessing homotopy commutativity of the diagram
    \begin{equation}
    \begin{tikzcd}
    & \mathrm{hofib}B^2J\arrow[d] \\
    \scrP(L_0,L_1)\arrow[r]\arrow[ur,"F_\theta+\widetilde{h}"] & B^2O
    \end{tikzcd}
    .
    \end{equation}
\end{proof}

Via precomposing with $BJ$, we have the following result.

\begin{cor}
Suppose $M$ is the plumbing of two cotangent bundles $T^*Q_i$ along $NC$, then we have a natural map
    \begin{equation}
    \map_*(C,BO)\to\mas^\theta.
    \end{equation}
\end{cor}

\begin{rem}
The reader may now wonder why we chose to work in the generality of stable spherical fibrations as opposed to using the more classical (but less general) virtual vector bundles. The answer is that the split injectivity of \eqref{eq:tau} witnessed by \eqref{eq:splinj} will fail as a result of also incorporating the $J$-homomorphism. Our main result would not have been as general, since it would hinge on the non-triviality of the $J$-homomorphism in a certain degree; this shows that stable spherical fibrations are, at least in certain scenarios, more natural to consider.
\end{rem}

\subsection{Parameterized Lagrangian Floer homotopy type}
We denote by 
    \begin{equation}
    \calH\calJ^\reg
    \end{equation}
the Kan complex whose: 0-simplices are regular admissible pairs, 1-simplices are regular admissible Floer continuation data connecting two pairs of regular admissible pairs, 2-simplices are regular admissible homotopies connecting the concatenation of regular admissible Floer continuation data and regular admissible Floer continuation data, and so on. The following result is standard.

\begin{lem}
$\calH\calJ^\reg$ is an acyclic Kan complex.
\end{lem}

In this subsection, we will construct our parameterized Lagrangian Floer homotopy type; this will be an $\infty$-functor 
    \begin{equation}
    \bbF^\sfr:\param\equiv\calH\calJ^\reg\times\mas\to\flow^\sfr.
    \end{equation}
    
Let $\bfH:\Delta^n\to\calH\calJ^\reg$ be an $n$-simplex, $n\geq1$. By an abuse of notation, we identify $\Delta^n$ with its geometric realization. Moreover, given 
    \begin{equation}
    \langle j_1\cdots j_k\rangle\subset\Delta^n, 0\leq j_1<\cdots <j_k\leq n,
    \end{equation}
we denote by 
    \begin{equation}
    \partial^{\langle j_1\cdots j_k\rangle}\Delta^n\subset\Delta^n
    \end{equation}
the associated boundary stratum. Associated to $\bfH$ is: an ordered $(n+1)$-tuple of regular admissible Floer data given by restricting to vertices,
    \begin{equation}
    \big(\bfH\vert_{\partial^{\langle 0\rangle}\Delta^n},\ldots,\bfH\vert_{\partial^{\langle n\rangle}\Delta^n}\big);
    \end{equation}
various regular admissible Floer continuation data given by restricting to edges,
    \begin{equation}
    \bfH\vert_{\partial^{\langle jk\rangle}\Delta^n}:\bfH\vert_{\partial^{\langle j\rangle}\Delta^n}\Rightarrow\bfH\vert_{\partial^{\langle k\rangle}\Delta^n},\;\; j<k;
    \end{equation}
various regular admissible homotopies connecting the concatentation of regular admissible Floer continuation data to regular admissible Floer continuation data given by restricting to 2-simplices,
    \begin{equation}
    \bfH\vert_{\partial^{\langle jk\ell\rangle}\Delta^n}:\bfH\vert_{\partial^{\langle jk\rangle}\Delta^n}\#\bfH\vert_{\partial^{\langle k\ell\rangle}\Delta^n}\Rightarrow\bfH\vert_{\partial^{\langle j\ell\rangle}\Delta^n},\;\;j<k<\ell;
    \end{equation}
and so on. The upshot is that, in a standard way, we may encode the various Floer continuation data, homotopies, homotopies of homotopies, etc. into an $(n-1)$ cube
    \begin{equation}
    \bfH^\mathrm{cube}:[0,1]^{n-1}\to\mathrm{Maps}\big([0,1],\calH\calJ^\reg\big),
    \end{equation}
whose coordinates are labeled by $\{1,\ldots,n-1\}$, such that a vertex of the cube (different from $(1,\ldots,1)$) determines a gluing of Floer continuation data 
    \begin{equation}
    \bfH\vert_{\partial^{\langle 0 j_1\rangle}\Delta^n}\#\bfH\vert_{\partial^{\langle j_1j_2\rangle}\Delta^n}\#\cdots\#\bfH\vert_{\partial^{\langle j_{k-1}j_k\rangle}\Delta^n}\#\bfH\vert_{\partial^{\langle j_kn\rangle}\Delta^n},\;\;0<j_1<\cdots<j_k<n,
    \end{equation}
where the coordinates of this vertex associated to all $j_{k'}$ labels are 0 and the coordinates of this vertex associated to all other labels are 1. Moreover, the vertex $(1,\ldots,1)$ determines 
    \begin{equation}
    \bfH\vert_{\partial^{\langle 0n\rangle}\Delta^n}.
    \end{equation}

Given any two $x\in\chi\big(L_0,L_1;\bfH\vert_{\partial^{\langle 0\rangle}\Delta^n}\big)$ and $y\in\chi\big(L_0,L_1;\bfH\vert_{\partial^{\langle n\rangle}\Delta^n}\big)$, we may consider the moduli space $\scrF^{\bfH^\mathrm{cube}}(x,y)$ consisting of tuples $(\vec{r},u)$, where $\vec{r}\in(0,1)^{n-1}$ and $u:Z\to M$ satisfies 
    \begin{equation}
    \begin{cases}
    \partial_su+J_{\bfH^\mathrm{cube}(\vec{r})}\Big(\partial_tu-X_{H_{\bfH^\mathrm{cube}(\vec{r})}}(u)\Big)=0 \\
    u(s,i)\in L_i \\
    \lim_{s\to-\infty}u(s,t)=x(t) \\
    \lim_{s\to+\infty}u(s,t)=y(t)
    \end{cases}
    .
    \end{equation}
where 
    \begin{equation}
    \Big(H_{\bfH^\mathrm{cube}(\vec{r})},J_{\bfH^\mathrm{cube}(\vec{r})}\Big)\equiv\bfH^\mathrm{cube}(\vec{r}).
    \end{equation}
Again, we see that $\scrF^{\bfH^\mathrm{cube}}(x,y)$ is a smooth manifold of dimension $\mu(x)-\mu(y)+n-1$ whose tangent bundle is classified by the index bundle $\ind D\overline{\partial}_\bfH$ of the family of surjective Fredholm operators 
    \begin{equation}
    D\overline{\partial}_\bfH\equiv\big\{D(\overline{\partial}_\bfH)_u:T(0,1)^{n-1}\oplus W^{1,2}(Z;u^*TM,u^*TL_i)\to L^2(Z;u^*TM)\big\},
    \end{equation}
given by linearizing the Floer equation. The following is essentially standard Gromov compactness.

\begin{prop}
Let $z$ denote an arbitrary time-1 Hamiltonian chord of the appropriate Hamiltonian. There is a natural Gromov-compactification $\bbF^{\bfH^\mathrm{cube}}(x,y)$ of $\scrF^{\bfH^\mathrm{cube}}(x,y)$ whose codimension 1 boundary strata are enumerated by natural gluing maps of the form
    \begin{align}
    \bbF^{\bfH\vert_{\partial^{\langle 0\rangle}\Delta^n}}(x,z)\times\bbF^{\bfH^\mathrm{cube}}(z,y)&\to\bbF^{\bfH^\mathrm{cube}}(x,y), \\
    \bbF^{\bfH^\mathrm{cube}}(x,z)\times\bbF^{\bfH\vert_{\partial^{\langle n\rangle}\Delta^n}}(z,y)&\to\bbF^{\bfH^\mathrm{cube}}(x,y), \\
    \dfrac{\coprod_z\Big(\bbF^{(\bfH\vert_{\partial^{\langle0\cdots j\rangle}\Delta^n})^\mathrm{cube}}(x,z)\times\bbF^{(\bfH\vert_{\partial^{\langle j\cdots n\rangle}\Delta^n})^\mathrm{cube}}(z,y)\Big)}{\sim}&\to\bbF^{\bfH^\mathrm{cube}}(x,y), \\
    \bbF^{(\bfH\vert_{\partial^{\langle 0\cdots\widehat{j}\cdots n\rangle}\Delta^n})^\mathrm{cube}}(x,y)&\to\bbF^{\bfH^\mathrm{cube}}(x,y),
    \end{align}
where $0<j<n$ and the equivalence relation identifies the images 
    \begin{equation}
    \begin{tikzcd}
    \bbF^{(\bfH\vert_{\partial^{\langle0\cdots j\rangle}\Delta^n})^\mathrm{cube}}(x,z')\times\bbF^{(\bfH\vert_{\partial^{\langle j\cdots n\rangle}\Delta^n})^\mathrm{cube}}(z',y) \\
    \bbF^{(\bfH\vert_{\partial^{\langle0\cdots j\rangle}\Delta^n})^\mathrm{cube}}(x,z)\times\bbF^{\bfH\vert_{\partial^{\langle j\rangle}\Delta^n}}(z,z')\times\bbF^{(\bfH\vert_{\partial^{\langle j\cdots n\rangle}\Delta^n})^\mathrm{cube}}(z,y)\arrow[d]\arrow[u]\\ \bbF^{(\bfH\vert_{\partial^{\langle0\cdots j\rangle}\Delta^n})^\mathrm{cube}}(x,z)\times\bbF^{(\bfH\vert_{\partial^{\langle j\cdots n\rangle}\Delta^n})^\mathrm{cube}}(z,y)
    \end{tikzcd}
    .
    \end{equation} 
\end{prop}

Moreover, the following result follows essentially from the proof of  \cite[Theorem 6.12, Proposition 7.4]{Lar21} (also, cf. \cite[Sections 6.2, 6.4, 6.5]{PS24b}).

\begin{prop}\label{prop:parameterized}
$\bbF^{\bfH^\mathrm{cube}}(x,y)$ is a stratified compact smooth manifold with corners whose tangent bundle is classified by (an extension of) $\ind D\overline{\partial}_\bfH$. Moreover, the natural decomposition of $T\bbF^{\bfH^\mathrm{cube}}(x,y)$ over a stratum into the direct sum of the tangent bundles of that stratum together with the collar directions, the natural decomposition of $\ind D\overline{\partial}_\bfH$ over a stratum into the direct sum of the index bundles of the families of linearized operators associated to that stratum, and the various isomorphism of the form $T\bbF^{\bfH^\mathrm{cube}}(x,y)\cong\ind D\overline{\partial}_\bfH$ are compatible.
\end{prop}

Now, let $\bfH\bfTheta:\Delta^n\to\param$ be an $n$-simplex, $n\geq1$. We denote by $\bfH$ resp. $\bfTheta$ the obvious $n$-simplex induced via projection:
    \begin{equation}
    \Delta^n\xrightarrow{\bfH\bfTheta}\param\to\calH\calJ^\reg\;\;\mathrm{resp.}\;\;\Delta^n\xrightarrow{\bfH\bfTheta}\param\to\mas.
    \end{equation}

\begin{prop}\label{prop:coherentorientations}
Suppose Assumption \ref{assu:main}, then the various $\bbF^{\bfH^\mathrm{cube}}(x,y)$ admit coherent $\bbS$-orientations.
\end{prop}

\begin{proof}
This is essentially the same as the proof of Proposition \ref{prop:unparamterizedframing}. By Proposition \ref{prop:parameterized}, $T\bbF^{\bfH^\mathrm{cube}}(x,y)$ is classified by a map
    \begin{equation}
    D\overline{\partial}_\bfH:\bbF^{\bfH^\mathrm{cube}}(x,y)\to\bbZ\times BO
    \end{equation}
related to the index bundles of the various families of linearized operators which is manifestly compatible with gluing. Again, we may consider the map
    \begin{equation}\label{eq:p1}
    D\overline{\partial}_\mathrm{aux}:\bbF^{\bfH^\mathrm{cube}}(x,y)\to BO
    \end{equation}
given by 
    \begin{equation}
    D\overline{\partial}_\bfH+T_{\frakF,y}-T_{\frakF,x}-\underline{\bbR}^{n-1}.
    \end{equation}
Meanwhile, we have the composition 
    \begin{equation}\label{eq:p2}
    \bbF^{\bfH^\mathrm{cube}}(x,y)\to\Omega\scrP(L_0,L_1)\to\Omega(U/O)\xrightarrow{\sim}\Omega^\infty KO\to BO\to BGL_1\bbS;
    \end{equation}
\cite[Proposition 3.17]{Bon25} shows \eqref{eq:p2} is equivalent to \eqref{eq:p1} (after postcomposing the latter with $BJ$). In a standard way, we have a projection map 
    \begin{equation}
    \bbF^{\bfH^\mathrm{cube}}(x,y)\to\Delta^n
    \end{equation}
(really, this projection exists after fixing an appropriate identification $[0,1]^n\cong\Delta^n$). We may choose a null-homotopy $\bfTheta^n$ of \eqref{eq:p2} by lifting $\bfTheta$ to $\bbF^{\bfH^\mathrm{cube}}(x,y)$ via the aforementioned projection; this is equivalent to saying we have an induced choice of virtual rank 0 real virtual bundle $I_\bfTheta(x,y)\to\bbF^{\bfH^\mathrm{cube}}(x,y)$, together with a trivialization as a stable spherical fibration, and a choice of isomorphism of real virtual bundles
    \begin{equation}
    T\bbF^{\bfH^\mathrm{cube}}(x,y)+\underline{\bbR}^{\mu(y)}\cong\underline{\bbR}^{\mu(x)}+\underline{\bbR}^{n-1}+I_\bfTheta(x,y).
    \end{equation}
Moreover, compatibility with inclusion of codimension 1 boundary strata follows by (1) canonical isomorphisms of the form 
    \begin{equation}
    \ind T_{\frakF,x}-\ind T_{\frakF,x}\cong0;
    \end{equation}
(2) the fact that (the extension to the Gromov-compactification of) $\ind D\overline{\partial}_\bfH$ is compatible with gluing; (3) the fact that the Bott isomorphism intertwines the $\bbE_\infty$-structure given by addition of index bundles and the $\bbE_1$-structure given by loop concatenation via the Eckmann-Hilton argument; and (4) the form of $\bfTheta^n$.
\end{proof}

The upshot is that we have a spherically framed flow $n$-simplex $\bbF^{\bfH\bfTheta}$ with object $(n+1)$-tuple
    \begin{equation}
    \big(\chi(L_0,L_1;H_0),\ldots,\chi(L_0,L_1;H_n)\big),
    \end{equation}
where 
    \begin{equation}
    \bfH\bfTheta\big(\partial^{\langle j\rangle}\Delta^n\big)=\big(H_j,J_j,\widetilde{\theta}_j\big),
    \end{equation}
morphism spaces $\bbF^{\bfH^\mathrm{cube}}(x,y)$, and spherical framing induced by Proposition \ref{prop:coherentorientations}. In the case that $\bfH\bfTheta:\Delta^0\to\param$ is a 0-simplex, we define $\bbF^{\bfH\bfTheta}$ to simply be $\bbF^{H,J,\theta}$, where $\bfH\bfTheta(\Delta^0)=(H,J,\theta)$. Finally, it is straightforward to see that if 
    \begin{equation}
    \bfH\bfTheta\to\partial^\sigma\bfH\bfTheta\equiv\bfH\bfTheta\vert_{\partial^\sigma\Delta^n}
    \end{equation}
is a face map, then we have a corresponding face map 
    \begin{equation}
    \bbF^{\bfH\bfTheta}\to\bbF^{\partial^\sigma\bfH\bfTheta}; 
    \end{equation}
these satisfy the natural associativity diagrams. In particular, we may define a morphism of semi-simplicial sets 
    \begin{equation}
    \bbF^\sfr:\param\to\flow^\sfr
    \end{equation}
which, on objects, sends an $n$-simplex $\bfH\bfTheta:\Delta^n\to\param$ to the spherically framed flow $n$-simplex $\bbF^{\bfH\bfTheta}$. By \cite[Theorem 1.4]{Tan18}, $\bbF^\sfr$ lifts to an $\infty$-functor.

\begin{defin}
The \emph{parameterized (spherically framed) Lagrangian Floer homotopy type} of $(L_0,L_1)$ is the $\infty$-functor 
    \begin{equation}
    \param\xrightarrow{\bbF^\sfr}\flow^\sfr\xrightarrow{\flow^\sfr(\unit,\cdot)}\mathrm{Mod}_{\frakR^\sfr},
    \end{equation}
denoted $\frakF^\sfr$.
\end{defin}

The following lemma follows since $\calH\calJ^\mathrm{reg}$ is an acyclic Kan complex.

\begin{lem}
The natural projection 
    \begin{equation}
    \param\to\mas
    \end{equation}
is an acyclic Kan fibration.
\end{lem}

\begin{cor}
We may identify $\frakF^\sfr$ with an $\infty$-functor
    \begin{equation}\label{eq:aux7}
    \mas\to\mathrm{Mod}_{\frakR^\sfr},
    \end{equation}
which, by an abuse of notation, we still denote by $\frakF^\sfr$.
\end{cor}

We end this section by showing that, as foreshadowed earlier, while $\frakF^\sfr$ may be a quite complicated object in general, in certain geometric situations we may be able to understand pullbacks of it to more manageable spaces. Consider the following result.

\begin{prop}\label{prop:pullback}
Suppose $M$ is the plumbing of two cotangent bundles $T^*Q_i$ along $NC$, then the $\infty$-functor
    \begin{equation}
    \map_*(C,BGL_1\bbS)\to\mas^\theta\xrightarrow{\frakF^\sfr}\mathrm{Mod}_{\frakR^\sfr}
    \end{equation}
(after choosing a suitable basepoint $\theta$) is the $\infty$-functor given on 0-simplices by 
    \begin{equation}
    \big(\Phi:C\to BGL_1\bbS\big)\mapsto C^\Phi\wedge\frakR^\sfr,
    \end{equation}
where $C^\Phi$ is the Thom spectrum associated to $\Phi$.
\end{prop}

\begin{rem}\label{rem:pullback}
If we further consider the $\infty$-functor
    \begin{multline}
    \map_*(C,BGL_1\bbS)_{\Phi_0}\to\map_*(C,BGL_1\bbS)\to\mas^\theta\xrightarrow{\frakF^\sfr}\mathrm{Mod}_{\frakR^\sfr}
    \end{multline}
where $\map_*(C,BGL_1\bbS)_{\Phi_0}$ is the connected component of some $\Phi_0$, Proposition \ref{prop:pullback} shows the pullback of $\frakF^\sfr$ is a parameterized spectrum, with base the connected component of $\Phi_0$, whose fiber is abstractly homotopy equivalent to $C^{\Phi_0}\wedge\frakR^\sfr$. However, as we shall see shortly, this abstract homotopy equivalence does not necessarily extend to families, i.e., the pullback of $\frakF^\sfr$ to the connected component of $\Phi_0$ does not necessarily factor through a point.
\end{rem}

\begin{proof}
A natural first step is to define what a ``suitable choice of basepoint $\theta$'' is in the statement of the current proposition. We see $T^*Q_i$ has a natural choice of stable $\bbR$-polarization given by the horizontal distribution, i.e., 
    \begin{equation}
    T(T^*Q_i)\cong\pi^*_iTQ_i\otimes_\bbR\underline{\bbC},
    \end{equation}
where $\pi_i:T^*Q_i\to Q_i$ is the standard projection. Moreover, the 0-section $\calO_i\subset T^*Q_i$ has a natural choice of framed brane structure with respect to this stable $\bbR$-polarization, i.e., $T\calO_i$ is homotopic to $\pi^*_iTQ_i$ through totally real subbundles of $T(T^*Q_i)$ via the constant homotopy. As in \cite[Subsection 7.3]{Bla24}, the stable $\bbR$-polarizations on $T^*Q_i$ naturally glue together to a stable $\bbR$-polarization on the plumbing $M$ and $(L_0,L_1)$ has a natural framed brane structure with respect to this stable $\bbR$-polarization; the corresponding $\theta\in\mas$ is our suitable choice of basepoint.

Now, we will investigate a bit of the behavior of the map from Lemma \ref{lem:BGL}. First, consider the map induced by freeness of the James construction $\Omega \Sigma C$ as an $\mathbb E_1$-group:
    \begin{equation}
    \Psi:\map_*(C,BGL_1\bbS)\to\map_*(\Omega\Sigma C,BGL_1\bbS).
    \end{equation}
Recall that, alternatively, $\Psi$ is obtained by (1) suspension-loop adjunction and (2) looping. There is a standard way of topologically embedding $C$ into $\Omega\Sigma C$; for any $c\in C$, we may define the loop 
    \begin{align}
    \gamma_c:[0,1]&\to\Sigma C\equiv\big(C\times[0,1]\big)/\sim \\
    t&\mapsto (c,t). \nonumber
    \end{align}
By unraveling the suspension-loop adjunction and looping, we see that 
    \begin{equation}
    \big(\Psi(\Phi)\big)(\gamma_c)=\Phi\vert_c.
    \end{equation}
A formal way of seeing this is that $\Omega \Sigma C$ is the free $\mathbb E_1$-algebra on $C$, so of course the induced map $\Omega \Sigma C \to BGL_1\bbS$ extends the original $C \to BGL_1\bbS$. Either way, in particular the image of the constant path at $c$ (which is a path connecting $L_0$ to $L_1$ that we will still denote by $c$) under the map
    \begin{equation}
    \scrP(L_0,L_1)\to\Omega\Sigma C
    \end{equation}
is $\gamma_c$; hence, the image of $c$ under the composition
    \begin{equation}\label{eq:compaux}
    \scrP(L_0,L_1)\to\Omega\Sigma C\xrightarrow{\Phi}BGL_1\bbS
    \end{equation}
is $\Phi\vert_c$. By Lemma \ref{lem:hofib}, the sum of $F_\theta$ and the postcomposition of \eqref{eq:compaux} with $BGL_1\bbS\to\mathrm{hofib}B^2J$ is equivalent to a choice of null-homotopy of the composition 
    \begin{equation}
    \scrP(L_0,L_1)\to B^2O\xrightarrow{B^2J}B^2GL_1\bbS.
    \end{equation}
After looping, this gives a choice of null-homotopy of the composition
    \begin{equation}\label{eq:compaux2}
    \Omega_{c,c'}\scrP(L_0,L_1)\simeq\Omega\scrP(L_0,L_1)\to BO\xrightarrow{BJ}BGL_1\bbS.
    \end{equation}
Thus, we have a choice of null-homotopy of \eqref{eq:compaux2} precomposed with the obvious map
    \begin{equation}
    \Omega_{c,c'}C\to\Omega_{c,c'}\scrP(L_0,L_1);
    \end{equation}
in particular, we see that this choice of null-homotopy is a homotopy from the composition
    \begin{equation}
    \Omega_{c,c'}C\to\Omega_{c,c'}\scrP(L_0,L_1)\simeq\Omega\scrP(L_0,L_1)\to BO\xrightarrow{BJ}BGL_1\bbS
    \end{equation}
to the constant map
    \begin{align}
    \Omega_{c,c'}C&\to BGL_1\bbS \\
    \Gamma_{c,c'}&\mapsto\Phi\vert_c-\Phi\vert_{c'}. \nonumber
    \end{align}

We finish the proof in the following manner. Let $\scrU_C\subset L_0$ be a tubular neighborhood of $C$. \cite[Section 7.3]{Bla24} shows that $L_1$, near $C$, can be identified with 
    \begin{equation}
    \big\{(p,df_p)\big\}\subset T^*\scrU_C,
    \end{equation}
where $f\in C^\infty(\scrU_C)$ is a smooth function with a global minimum along $C$ and no other critical points. The upshot is that, after Morsifying $f$ into $f'$ via a $C^2$-small perturbation, Subsection 6.2 of \emph{loc. cit.} combined with Proposition \ref{prop:sfrfr} shows we have a homotopy equivalence 
    \begin{equation}
    \frakF^{H,J,\theta}\simeq\Sigma^\infty_+C\wedge\frakR^\sfr
    \end{equation}
for any choice of regular admissible pair $(H,J)$. This follows because (1) the Floer trajectories, for a suitable choice of $(H,J)$, are actually in bijection with Morse trajectories of $f'$ and (2) the linearization of the Floer equation resp. a Floer abstract cap, in this situation, is intimately related to the linearization of the negative gradient flow equation resp. a Morse abstract cap; this essentially goes back to Floer \cite{Flo88c}. A suitable modification of this (i.e., accounting for the twist in the framing determined by $\Phi\in\map_*(C,BGL_1\bbS)$) shows we have a homotopy equivalence 
    \begin{equation}
    \frakF^{H,J,\theta_\Phi}\simeq C^\Phi\wedge\frakR^\sfr,
    \end{equation}
where $\theta_\Phi$ is the image of $\Phi$ under the map from Lemma \ref{lem:BGL}.\footnote{Again, this can also be shown in the Cohen-Jones-Segal framework, cf. \cite[Theorem 5.9]{Bon24} and \cite[Theorem 3.11]{CK23}.} This follows because, in the definition of the framing for $\bbF^{H,J,\theta_\Phi}$, we considered the map 
    \begin{equation}
    \bbF^{H,J}(x,y)\to\Omega_{x,y}\scrP(L_0,L_1),
    \end{equation}
and the observation that the Floer trajectories of $(H,J)$ are in bijection with the Morse trajectories of $f'$ implies we have a factorization of this map:
    \begin{equation}
    \bbF^{H,J}(x,y)\to\Omega_{x(0),y(0)}C\to\Omega_{x,y}\scrP(L_0,L_1).
    \end{equation}
But, we have just investigated how $\theta_\Phi$ behaves on the image of the second map. In particular, the spherical framing for $\bbF^{H,J,\theta_\Phi}$ is given by an isomorphism of real virtual bundles of the form
    \begin{equation}\label{eq:coherentisoaux}
    T\bbF^{H,J}(x,y)+\underline{\bbR}+\underline{\bbR}^{\mu(y)}+\underline{\Phi}\vert_{y(0)}\cong\underline{\bbR}^{\mu(x)}+\underline{\Phi}\vert_{x(0)}+I(x,y),
    \end{equation}
where $I(x,y)\to \bbF^{H,J}(x,y)$ is a virtual rank 0 real virtual bundle together with a trivialization as a stable spherical fibration, which is compatible with gluing. But, after using the identification between the Floer trajectories of $(H,J)$ and the Morse trajectories of $f'$, this is precisely the spherical framing for the Morse flow category of $f'$ twisted by $\Phi$; the proposition follows.
\end{proof}

\section{Proof of main result}
\subsection{Preliminary considerations}
In this section, we will work in the case that $M$ is the plumbing of two cotangent bundles $T^*Q_i$ along $NC$. In Proposition \ref{prop:pullback}, we have identified the $\infty$-functor 
    \begin{equation}
    \map_*(C,BGL_1\bbS)\to\mas^\theta\xrightarrow{\frakF^\sfr}\mathrm{Mod}_{\frakR^\sfr}
    \end{equation}
with a more familiar one. Let $\map_*(C,BGL_1\bbS)_0$ be the connected component of the trivial map, i.e., the connected component of the trivial stable spherical fibration $C\otimes\bbS$. As noted in Remark \ref{rem:pullback}, the $\infty$-functor 
    \begin{multline}\label{eq:non-zeroaux1}
    \map_*(C,BGL_1\bbS)_0\to\map_*(C,BGL_1\bbS)\to\mas^\theta\xrightarrow{\frakF^\sfr}\mathrm{Mod}_{\frakR^\sfr}
    \end{multline}
is identified with the $\infty$-functor 
    \begin{equation}\label{eq:non-zeroaux2}
    \th^\sfr:\map_*(C,BGL_1\bbS)_0\to B\aut\big(\Sigma^\infty_+C\wedge\frakR^\sfr\big),
    \end{equation}
which, on 0-simplices, sends $\Phi$ to $C^\Phi\wedge\frakR^\sfr$. Here, $B\aut\big(\Sigma^\infty_+C\wedge\frakR^\sfr\big)\subset\mathrm{Mod}_{\frakR^\sfr}$ is the full $\infty$-subgroupoid of $\frakR^\sfr$-modules which are homotopy equivalent to $\Sigma^\infty_+C\wedge\frakR^\sfr$. The purpose of this subsection is to prove the following result.

\begin{prop}\label{prop:non-zero}
Suppose $C$ is: connected, orientable, and positive-dimensional; then \eqref{eq:non-zeroaux1} does not factor through a point.
\end{prop}

Clearly, proving the proposition is equivalent to proving \eqref{eq:non-zeroaux2} does not factor through a point. Moreover, it is straightforward to see that \eqref{eq:non-zeroaux2} is identified with the composition of $\infty$-functors
    \begin{equation}
    \map_*(C,BGL_1\bbS)_0\xrightarrow{\th}B\aut(\Sigma^\infty_+C)\xrightarrow{(\cdot)\wedge\frakR^\sfr}B\aut\big(\Sigma^\infty_+C\wedge\frakR^\sfr\big),
    \end{equation}
where $\th$ is the $\infty$-functor which, on 0-simplices, sends $\Phi$ to $C^\Phi$ and $B\aut\big(\Sigma^\infty_+C)\subset\spectra$ is the full $\infty$-subgroupoid of $\spectra$ which are homotopy equivalent to $\Sigma^\infty_+C$. First, we will investigate $\th$. Since $C$ is pointed at $c_0$, we have a canonical splitting 
    \begin{equation}
    \Sigma^\infty_+C\simeq\Sigma^\infty C\vee\bbS.
    \end{equation}
Consider the automorphism
    \begin{equation}
    \Omega\th(\phi):\Sigma^\infty C\vee\bbS\to\Sigma^\infty C\vee\bbS,
    \end{equation}
where
    \begin{equation}
    \phi\in\map_*(C,GL_1\bbS)\simeq\Omega\map_*(C,BGL_1\bbS)_0;
    \end{equation}
note, we have that 
    \begin{equation}
    \phi(c_0)=(\identity:\bbS\to\bbS).
    \end{equation}
We may write $\Omega\th(\phi)$ as a ``block matrix'' in the ``basis'' $\Sigma^\infty C\vee\bbS$:
    \begin{equation}
    \begin{bmatrix}
    \aut(\bbS) & \map(\Sigma^\infty C,\bbS) \\
    \map(\bbS,\Sigma^\infty C) & \aut(\Sigma^\infty C)
    \end{bmatrix}
    .
    \end{equation}
Again, since $C$ is pointed at $c_0$, we have a retraction
    \begin{equation}
    *\vee\bbS\to\Sigma^\infty C\vee\bbS\to*\vee\bbS.
    \end{equation}
In particular, we have that (1) $\Omega\th(\phi)$ will preserve the subspectrum $*\vee\bbS$ and (2) the top-left entry of $\Omega\th(\phi)$'s block matrix is the identity, i.e., $\Omega\th(\phi)$ is of the form
    \begin{equation}
    \begin{bmatrix}
    \identity & \map(\Sigma^\infty C,\bbS) \\
    0 & \aut(\Sigma^\infty C)
    \end{bmatrix}
    .
    \end{equation}
We define the map
	\begin{equation}
	\aut\big(\Sigma^\infty_+C\big)\xrightarrow{\mathrm{t.r.}}\map(\Sigma^\infty C,\bbS)
	\end{equation}
via taking the top-right entry of a block matrix.

\begin{lem}
Suppose $\varphi\in\map(\Sigma^\infty C,\bbS)$, then there exists a canonical choice of $\phi\in\map_*(C,GL_1\bbS)$ such that $\Omega\th(\phi)$ is of the form
    \begin{equation}
    \begin{bmatrix}
    \identity & \varphi \\
    0 & \aut(\Sigma^\infty C)
    \end{bmatrix}
    .
    \end{equation}
\end{lem}

\begin{proof}
Since $C$ is connected and $\map(\bbS,\bbS)\simeq\Omega^\infty\bbS$ is pointed at the trivial map, we see that $\varphi$ determines an element
    \begin{equation}
    \widetilde{\varphi}\in\map_*\big(C,\map(\bbS,\bbS)_0\big)\simeq\map(\Sigma^\infty C,\bbS)
    \end{equation}
via suspension-loop adjunction. Moreover, it is known that adding the identity element of an $\bbE_\infty$-ring is a homotopy equivalence from the connected component of the zero element to the connected component of the identity element, i.e., we have
    \begin{equation}
    \identity+\widetilde{\varphi}\in\map_*(C,GL_1\bbS),
    \end{equation}
where we abuse notation and also denote by $\identity:C\to GL_1\bbS$ the pointed map $c\mapsto(\identity:\bbS\to\bbS)$. If we consider the automorphism
    \begin{equation}
    \Omega\th(\identity+\widetilde{\varphi}):\Sigma^\infty_+ C\simeq C\otimes\bbS\to \Sigma^\infty_+C\simeq C\otimes\bbS,
    \end{equation}
then we observe $\Omega\th(\identity+\widetilde{\varphi})$ maps the fiber $\bbS\vert_c$ to the fiber $\bbS\vert_c$ via $\identity+\widetilde{\varphi}(c)$.

We now slightly digress to recall two standard facts about additive categories. 
\begin{enumerate}
\item Let $X_i$ be two objects and $\iota_i$ resp. $\pi_i$ the obvious inclusion resp. projection maps associated to the direct sum $X_0\oplus X_1$. We have that 
    \begin{equation}
    \iota_0\circ\pi_0+\iota_1\circ\pi_1=\identity.
    \end{equation}

\item Let $Y$ be another object. We have that 
    \begin{equation}
    \pi_i^*:\map(X_i,Y)\to\map(X_0\oplus X_1,Y)
    \end{equation}
is a split-injection.
\end{enumerate}

We now return to the proof at hand. We may write 
    \begin{equation}
    \mathrm{t.r.}\circ\Omega\th(\identity+\widetilde{\varphi})
    \end{equation}
as the composition 
    \begin{equation}
    \Sigma^\infty C\vee*\xrightarrow{\iota_{\Sigma^\infty C}}\Sigma^\infty C\vee\bbS\xrightarrow{\identity+\widetilde{\varphi}'}\Sigma^\infty C\vee\bbS\xrightarrow{\pi_\bbS}\bbS, 
    \end{equation}
where the notation $\identity+\widetilde{\varphi}'$ indicates the map acting fiberwise via $\identity+\widetilde{\varphi}(c)$. By item (2) of the digression, it suffices to prove 
    \begin{equation}
    \pi_\bbS\circ(\identity+\widetilde{\varphi}')\circ\iota_{\Sigma^\infty C}\circ\pi_{\Sigma^\infty C}=\varphi\circ\pi_{\Sigma^\infty C}.
    \end{equation}
By item (1) of the digression, we may compute:
    \begin{align}
    \pi_\bbS\circ(\identity+\widetilde{\varphi}')\circ\iota_{\Sigma^\infty C}\circ\pi_{\Sigma^\infty C}&=(\pi_\bbS+\pi_\bbS\circ\widetilde{\varphi}')\circ(\identity-\iota_\bbS\circ\pi_\bbS) \\
    &=\pi_\bbS+\pi_\bbS\circ\widetilde{\varphi}'-\pi_\bbS\circ\iota_\bbS\circ\pi_\bbS-\pi_\bbS\circ\widetilde{\varphi}'\circ\iota_\bbS\circ\pi_\bbS \nonumber \\
    &=\pi_\bbS\circ\widetilde{\varphi}', \nonumber
    \end{align}
where we used the fact (1) the first and third term cancel and (2) $\widetilde{\varphi}'\circ\iota_\bbS=0$ since $\widetilde{\varphi}'(c_0)=0$. Finally, we have the equality 
    \begin{equation}
    \pi_\bbS\circ\widetilde{\varphi}'=\varphi\circ\pi_{\Sigma^\infty C},
    \end{equation}
since $\widetilde{\varphi}$ and $\varphi$ are related via suspension-loop adjunction. We take $\phi=\identity+\widetilde{\varphi}$.
\end{proof}

We will denote the following composition by $\tau$:
    \begin{equation}\label{eq:tau}
    \map_*(C,GL_1\bbS)\xrightarrow{\Omega\th}\aut(\Sigma^\infty_+C)\xrightarrow{\mathrm{t.r.}}\map(\Sigma^\infty C,\bbS)\simeq\Omega^\infty(\frakD\Sigma^\infty C),
    \end{equation}
where $\frakD\Sigma^\infty C$ is the Spanier-Whitehead dual of $\Sigma^\infty C$. Moreover, we consider the homotopy equivalence
    \begin{multline}
    \psi:\map_*(C,GL_1\bbS)\xrightarrow{\sim}\map_*\big(C,\map(\bbS,\bbS)_0\big)\simeq \\
    \map_*(C,\Omega^\infty\bbS)\simeq\map(\Sigma^\infty C,\bbS)\simeq\Omega^\infty(\frakD\Sigma^\infty C),
    \end{multline}
where the first homotopy equivalence comes from subtracting the identity element.

\begin{cor}\label{cor:identity}
The composition 
    \begin{equation}\label{eq:splinj}
    \tau\circ\psi^{-1}:\Omega^\infty(\frakD\Sigma^\infty C)\to\Omega^\infty(\frakD\Sigma^\infty C)
    \end{equation}
is the identity.
\end{cor}

\begin{proof}
We have that 
    \begin{equation}
    \psi^{-1}(\varphi)=\identity+\widetilde{\varphi}.
    \end{equation}
Now, by the lemma, 
    \begin{equation}
    \tau(\identity+\widetilde{\varphi})=\mathrm{t.r.}\circ\Omega\th(\identity+\widetilde{\varphi})=\varphi.
    \end{equation}
\end{proof}

We now return to $\th^\sfr$. Again, we may consider the automorphism 
    \begin{equation}
    \Omega\th^\sfr(\phi):(\Sigma^\infty C\wedge\frakR^\sfr)\vee\frakR^\sfr\to(\Sigma^\infty C\wedge\frakR^\sfr)\vee\frakR^\sfr;
    \end{equation}
moreover, we may write $\Omega\th^\sfr(\phi)$ as a block matrix:
    \begin{equation}
    \begin{bmatrix}
    \aut(\frakR^\sfr) & \mathrm{Mod}_{\frakR^\sfr}(\Sigma^\infty C\wedge\frakR^\sfr,\frakR^\sfr) \\
    \mathrm{Mod}_{\frakR^\sfr}(\frakR^\sfr,\Sigma^\infty C\wedge\frakR^\sfr) & \aut(\frakR^\sfr)
    \end{bmatrix}
    .
    \end{equation}
Consider the composition 
    \begin{multline}
    \map_*(C,BGL_1\bbS)_0\xrightarrow{\Omega\th^\sfr}\aut(\Sigma^\infty C_+\wedge\frakR^\sfr)\xrightarrow{\mathrm{t.r.}} \\
    \mathrm{Mod}_{\frakR^\sfr}(\Sigma^\infty C\wedge\frakR^\sfr,\frakR^\sfr).
    \end{multline}
Clearly, this map is equal to the composition
    \begin{multline}
    \map_*(C,BGL_1\bbS)_0\xrightarrow{\Omega\th}\aut(\Sigma^\infty C_+)\xrightarrow{\mathrm{t.r.}} \\
    \map(\Sigma^\infty C,\bbS)\to\mathrm{Mod}_{\frakR^\sfr}(\Sigma^\infty C\wedge\frakR^\sfr,\frakR^\sfr).
    \end{multline}
By Lemma \ref{cor:identity}, the composition of the first two maps is the identity, i.e., $\mathrm{t.r.}\circ\Omega\th^\sfr$ is equal to the map 
    \begin{equation}
    \map(\Sigma^\infty C,\bbS)\cong\mathrm{Mod}_{\frakR^\fr}(\Sigma^\infty C,\frakR^\fr)\to\mathrm{Mod}_{\frakR^\sfr}(\Sigma^\infty C\wedge\frakR^\sfr,\frakR^\sfr).
    \end{equation}
By Theorem \ref{thm:appendixb}, this map induces a non-zero map of homotopy groups in infinitely many negative degrees if there exists a non-torsion element in $H_*(C;\bbZ)$. The upshot is that $\Omega\th^\sfr$ cannot factor through a point if there exists a non-torsion element in $H_*(C;\bbZ)$; this proves Proposition \ref{prop:non-zero}.

\subsection{Lagrangian intersections}\label{subsec:lagrangianintersections}

We are now ready to prove our main result.

\begin{proof}[Proof of Theorem \ref{thm:main}]
We suppose for contradiction that there exists a compactly supported Hamiltonian $H\in C^\infty(I\times M)$ such that 
    \begin{equation}
    L_0\cap\phi^{-1}_{H_t}(L_1)=\{p_\alpha\}_{\alpha\in A},
    \end{equation}
where each (possibly degenerate) intersection point $p_\alpha$ has the same action. Note, we may assume $A$ is a finite indexing set. Given any $x_\alpha\in\chi(L_0,L_1;H)$, let $\scrU_\alpha\subset M$ be a contractible neighborhood of $x_\alpha$ such that (1) $\scrU_\alpha\cap\scrU_{\alpha'}\neq\emptyset$ if and only if $\alpha=\alpha'$ and (2) $L_i\cap\scrU_\alpha\cong\bbR^{\dim L_i}$. We may assume there exists $f_\alpha\in C^\infty(L_0\cap\scrU_\alpha)$ such that 
    \begin{equation}
    \phi^{-1}_{H_t}(L_1)\cap\scrU_\alpha=\big\{(p,d(f_\alpha)_p\big\}\subset T^*(L_0\cap\scrU_\alpha).
    \end{equation}
Let $J$ be an admissible almost complex structure and $f_\alpha'$ the Morsification of $f_\alpha$ via a $C^2$-small perturbation. In particular, we may perturb $(H,J)$ into $(H',J')$, via a $C^2$-small perturbation compactly supported in a neighborhood of each $x_\alpha$, such that $(H',J')$ satisfies the following properties:
\begin{itemize}
\item $H'$ is non-degenerate,
\item each $x'\in\chi(L_0,L_1;H')$ satisfies $x'\in\scrU_\alpha$ for some $\alpha$,
\item $\widetilde{\scrF}^{H',J'}(x',y')\neq\emptyset$ if and only if there exists $\alpha$ such that (1) $x',y'\in\scrU_\alpha$ and (2) every Floer trajectory $u$ connecting $x'$ to $y'$ satisfies $u(Z)\subset\scrU_\alpha$,
\item and $\phi^{-1}_{H_t'}(L_1)\cap\scrU_\alpha=\big\{(p,d(f_\alpha')_p\big\}\subset T^*(L_0\cap\scrU_\alpha)$.
\end{itemize}
In fact, if $\widetilde{\scrF}^{H',J'}(x',y')\neq\emptyset$, then it is naturally identified to the moduli space of Morse trajectories connecting two critical points of $f_\alpha'$. The upshot is that the Floer theory of $(L_0,L_1)$ is entirely local, i.e., it reduces to the Morse theory of each $f_\alpha'$. Without loss of generality, we may assume there exists a path $\Gamma_\alpha$ connecting the constant path at $c_0$ to $x_\alpha$.\footnote{This is because the Floer homotopy type splits as a wedge sum of Floer homotopy types over the various components of $\scrP(L_0,L_1)$. Since $L_0$ intersects $L_1$ cleanly in $C$ (which is connected), it immediately follows that the Floer homotopy type over every component of $\scrP(L_0,L_1)$, besides the component containing $c_0$, is the trivial spectrum.} Now, by our discussion above, the map 
    \begin{equation}
    \bbF^{H',J'}(x',y')\to\Omega_{x',y'}\scrP(L_0,L_1)\simeq\Omega\big(\scrP(L_0,L_1),c_0\big),
    \end{equation}
where $\Omega\big(\scrP(L_0,L_1),c_0\big)$ indicates this based loop space is pointed at $c_0$, factors, up to homotopy, as the composition: 
    \begin{multline}\label{eq:factorizationaux}
    \bbF^{H',J'}(x',y')\to\Omega_{x',y'}\scrP(L_0\cap\scrU_a,L_1\cap\scrU_\alpha)\simeq \\
    \Omega\big(\scrP(L_0\cap\scrU_a,L_1\cap\scrU_\alpha),x_\alpha\big)\to\Omega\big(\scrP(L_0,L_1),x_\alpha\big)\to\Omega\big(\scrP(L_0,L_1),c_0\big),
    \end{multline}
where the last map is the obvious map given via conjugation by $\Gamma_\alpha$. We observe that, since $\scrU_\alpha$ and $L_i\cap\scrU_\alpha$ are contractible, we have a homotopy equivalence
    \begin{equation}
    \Omega\big(\scrP(L_0\cap\scrU_a,L_1\cap\scrU_\alpha),x_\alpha\big)\simeq\{x_\alpha\}.
    \end{equation}
The upshot is that the composition 
    \begin{equation}
    \bbF^{H',J'}(x',y')\to\Omega\big(\scrP(L_0,L_1),c_0\big)\to BO\to BGL_1\bbS
    \end{equation}
is canonically null-homotopic if we only consider null-homotopies of \eqref{eq:assu} in the image of the map 
    \begin{equation}
    \map_*(C,BGL_1\bbS)_0\to\mas^\theta;
    \end{equation}
this is because (1) such null-homotopies are canonically null-homotopic at $c_0$, (2) the path $\Gamma_\alpha$ induces a canonical null-homotopy at $x_\alpha$, and (3) we have the factorization \eqref{eq:factorizationaux}. The upshot is that the $\infty$-functor
    \begin{equation}
    \map_*(C,BGL_1\bbS)_0\to\map_*(C,BGL_1\bbS)\to\mas^\theta\xrightarrow{\frakF^\sfr}\mathrm{Mod}_{\frakR^\sfr}
    \end{equation}
factors through the constant functor 
    \begin{align}
    *&\to\mathrm{Mod}_{\frakR^\sfr} \\
    *&\mapsto\bigvee_{\alpha\in A}\Sigma^\infty\Sigma^{d_\alpha}\calC_{f_\alpha}\wedge\frakR^\sfr, \nonumber
    \end{align}
where $\calC_{f_\alpha}$ is the Conley index at $p_\alpha$ of $f_\alpha$ (i.e., the space whose stable homotopy type is determined by the Morse homotopy type of $f_\alpha'$) and $d_\alpha\in\bbZ$ are integers determined by $\theta$. But this is an immediate contradiction to Proposition \ref{prop:non-zero}.
\end{proof}

\appendix
\section{Stable cohomotopy classes in negative degrees}\label{appendix:lemma}
Recall, a spectrum $\frakX$ is called \emph{Moore} if (1) $\pi_*\frakX$ is bounded below and (2) $H_*(\frakX;\bbZ)$ is concentrated in degree 0. Some standard facts about Moore spectra are the following. 
\begin{itemize}
\item For any abelian group $A$, there exists a Moore spectrum $\frakM A$ such that $H_0(\frakM A;\bbZ)\cong A$.
\item Any two Moore spectra with isomorphic $H_0(-;\bbZ)$ are isomorphic in $\ho\spectra$.
\item If $A_1,\ldots,A_k$ are abelian groups, then 
    \begin{equation}
    \frakM(A_1\oplus\cdots\oplus A_k)\cong\frakM A_1\vee\cdots\vee\frakM A_k
    \end{equation}
in $\ho\spectra$.
\end{itemize}
Some quick examples of Moore spectra are $\frakM\bbZ=\bbS$ and $\frakM\bbZ_{(p)}\cong\bbS_{(p)}$, where $p\in\bbZ$ is a prime number.

\begin{thm}\label{thm:appendixa}
Let $\frakX$ be a finite spectrum. For a sufficiently large prime $p$, the $p$-localization $\frakX_{(p)}\equiv\bbS_{(p)}\wedge\frakX$ becomes isomorphic to a direct sum of shifted Moore spectra in $\ho\spectra$.
\end{thm}

We begin with a preliminary lemma. 

\begin{lem}
Let $\frakX$ and $\frakY$ be two finite spectra. For any sufficiently large prime $p$, we have a chain of isomorphisms
    \begin{equation}
    \ho\spectra(\frakX,\frakY)_{(p)}\cong\ho\spectra\big(\frakX_{(p)},\frakY_{(p)}\big)\cong R\hom\big(H_*(\frakX;\bbZ_{(p)}),H_*(\frakY;\bbZ_{(p)})\big),
    \end{equation}
where $R\hom$ denotes the derived hom in the derived category of $\bbZ_{(p)}$-modules.
\end{lem}

\begin{proof}
Recall, for odd $p$, the first $p$-torsion element in $\pi_*\bbS$ appears in degree $2p-3$; hence, the map 
    \begin{equation}
    \bbS_{(p)}\to H\bbZ_{(p)}
    \end{equation}
is an isomorphism on homotopy groups in degree at most $2p-4$ and its cofiber $\frakQ$ is $(2p-4)$-connected. Consider the long exact sequence 
    \begin{multline}
    \cdots\to\ho\spectra\big(\frakX_{(p)},\Sigma^{-1}\frakQ\wedge\frakY_{(p)}\big)\to\ho\spectra\big(\frakX_{(p)},\frakY_{(p)}\big)\to \\
    \ho\spectra\big(\frakX_{(p)},H\bbZ_{(p)}\wedge\frakY_{(p)}\big)\to\ho\spectra\big(\frakX_{(p)},\frakQ\wedge\frakY_{(p)}\big)\to\cdots.
    \end{multline}
For sufficiently large $p$, $\Sigma^{-1}\frakQ\wedge\frakY_{(p)}$ has connectivity larger than the dimension of the top cell of $\frakX$, therefore the first and last term in the portion of the long exact sequence above vanish, i.e., we have an isomorphism
    \begin{equation}
    \ho\spectra\big(\frakX_{(p)},\frakY_{(p)}\big)\cong\ho\spectra\big(\frakX_{(p)},H\bbZ_{(p)}\wedge\frakY_{(p)}\big).
    \end{equation}
By the induction-restriction adjunction, we have an isomorphism 
    \begin{equation}
    \ho\mathrm{Mod}_{H\bbZ_{(p)}}\big(H\bbZ_{(p)}\wedge\frakX_{(p)},H\bbZ_{(p)}\wedge\frakY_{(p)}\big)\cong\ho\spectra\big(\frakX_{(p)},H\bbZ_{(p)}\wedge\frakY_{(p)}\big).
    \end{equation}
By the stable Dold-Kan correspondence, we have an isomorphism 
    \begin{equation}
    R\hom\big(C_*(\frakX;\bbZ_{(p)}),C_*(\frakY;\bbZ_{(p)})\big)\cong\ho\mathrm{Mod}_{H\bbZ_{(p)}}\big(H\bbZ_{(p)}\wedge\frakX_{(p)},H\bbZ_{(p)}\wedge\frakY_{(p)}\big).
    \end{equation}
Finally, the lemma follows because a finitely-generated finitely-supported chain complex of free modules over a PID is quasi-isomorphic to its homology in the derived category.
\end{proof}

\begin{proof}[Proof of Theorem \ref{thm:appendixa}]
The integral homology of $\frakX_{(p)}$ splits as a direct sum of graded $\bbZ_{(p)}$-modules, each supported in a single degree. Let $\iota_k$ resp. $\pi_k$ denote the obvious inclusion resp. projection maps from resp. to the summand in degree $k$. These maps satisfy 
    \begin{equation}
    \pi_k\circ i_k=\identity,\;\;\sum_k\iota_k\circ\pi_k=\identity
    \end{equation}
as maps of graded abelian groups. These maps induce maps with the same relations in the derived category of $\bbZ_{(p)}$-modules. By the lemma, these maps induce maps with the same relations in $\ho\spectra$. Finally, since $\ho\spectra$ is an additive category, these maps present $\frakX_{(p)}$ as a wedge sum of pieces, each having integral homology concentrated in a single degree. 
\end{proof}

\begin{cor}\label{cor:appendix}
Suppose $C$ is a: connected, orientable, and positive-dimensional; closed smooth manifold, then there exists a non-trivial reduced stable cohomotopy class of $C$ in non-positive degree.
\end{cor}

\begin{proof}
Since $C$ is orientable, we have that $H_{\dim C}(C;\bbZ)\cong\bbZ$. In particular, for any $p$, $H_{\dim C}(\Sigma^\infty C_{(p)};\bbZ)\cong\bbZ_{(p)}$. Theorem \ref{thm:appendixa} implies $\Sigma^\infty C_{(p)}$, for $p$ a sufficiently large prime, splits as a wedge sum of shifted Moore spectra with at least one $\bbS_{(p)}$ summand; the lemma follows since $\bbS_{(p)}$ has infinitely many non-zero stable cohomotopy groups in negative degrees (hence, so do $\Sigma^\infty C_{(p)}$ and $\Sigma^\infty C$).
\end{proof}

\section{Spherically framed versus framed bordism}\label{appendix:sfrfr}

In this appendix, we will elucidate the relationship between spherically framed bordism and (classical) framed bordism as it applies to the present article. Recall, we denote by $\frakR^\sfr$ the endomorphism ring spectrum $\flow^\sfr(\unit,\unit)$. Moreover, we will denote by $\frakR^\fr$ the endomorphism ring spectrum $\flow^\fr(\unit,\unit)$. We observe that there is an obvious forgetful map 
    \begin{equation}
    \frakR^\fr\to\frakR^\sfr.
    \end{equation}

\begin{defin}
Let $\bullet\in\{\sfr,\fr\}$. We define $\Omega^\bullet_*(\cdot)$ to be the multiplicative homology theory of $\bullet$-bordism defined in the usual way:
    \begin{equation}
    \Omega^\bullet_j(X)\equiv\big\{M\to X:\mathrm{\emph{$M$ closed smooth $j$-manifold with $\bullet$-structure}}\big\}/\mathrm{\emph{$\bullet$-bordism}}.
    \end{equation}
We again observe that there is an obvious forgetful map
    \begin{equation}
    \Omega^\fr_*(\cdot)\to\Omega^\sfr_*(\cdot).
    \end{equation}
Finally, there is an analogously defined multiplicative cohomology theory $\Omega^*_\bullet(\cdot)$ for finite spectra defined using Spanier-Whitehead duality; we have natural isomorphisms 
    \begin{equation}
    \Omega^\bullet_*(\cdot)\cong\Omega^{-*}_\bullet(\cdot).
    \end{equation}
\end{defin}

\begin{lem}\label{lem:appendixinjective}
The obvious forgetful map 
    \begin{equation}
    \Omega^\fr_*(\cdot)\to\Omega^\sfr_*(\cdot)
    \end{equation}
is a split-injection.
\end{lem}

\begin{rem}
This map is not surjective; for instance, the spherical framed bordism class of a $K3$ surface is not in its image.
\end{rem}

\begin{proof}
The proof follows since we may still perform the Pontryagin-Thom construction on a closed smooth manifold whose tangent bundle is trivial as a stable spherical fibrationa, i.e., we may define a map
    \begin{equation}
    \Omega^\sfr_*(X)\to\pi^\mathrm{st}_*(X)
    \end{equation}
such that the composition
    \begin{equation}
    \Omega^\fr_*(X)\to\Omega^\sfr_*(X)\to\pi^\mathrm{st}_*(X),
    \end{equation}
which is an isomorphism, is the usual Pontryagin-Thom construction.
\end{proof}

The following conjecture seems sensible. 

\begin{conj}\label{conj:sfrfr}
There are natural isomorphisms of multiplicative homology theories compatible with the obvious forgetful maps:
    \begin{equation}
    \begin{tikzcd}
    \frakR^\fr_*(\cdot)\arrow[r]\arrow[d,"\sim"] & \frakR^\sfr_*(\cdot)\arrow[d,"\sim"] \\
    \Omega^\fr_*(\cdot)\arrow[r] & \Omega^\sfr_*(\cdot)
    \end{tikzcd}
    .
    \end{equation}
Moreover, there are analogous natural isomorphisms of multiplicative cohomology theories compatible with the obvious forgetful maps.
\end{conj}

Note, the natural isomorphism $\frakR^\fr_*(\cdot)\to\Omega^\fr_*(\cdot)$ has already been shown by Abouzaid-Blumberg via showing there is a homotopy equivalence of ring spectra $\bbS\simeq\frakR^\fr$, cf. \cite[Corollary 8.12]{AB24}. 

\begin{conj}
There is a homotopy equivalence of ring spectra 
    \begin{equation}
    \frakR^\sfr\simeq\Sigma^\infty\mathrm{hofib}BJ,
    \end{equation}
$BJ$ is the 1-fold delooping of the $J$-homomorphism.
\end{conj}

We observe that the conclusion of Conjecture \ref{conj:sfrfr} is essentially tautological when considered at a point.

\begin{prop}\label{prop:appendixpoint}
We have a commutative diagram
    \begin{equation}
    \begin{tikzcd}
    \pi_j(\frakR^\fr)\cong\frakR^\fr_j(*)\arrow[r]\arrow[d,"\sim"] & \pi_j(\frakR^\sfr)\cong\frakR^\sfr_j(*)\arrow[d,"\sim"] \\
    \Omega^\fr_j(*)\arrow[r] & \Omega^\sfr_j(*)
    \end{tikzcd}
    .
    \end{equation}
\end{prop}

\begin{proof}
We have that 
    \begin{equation}
    \pi_j(\frakR^\bullet)=\pi_j\flow^\bullet(\unit,\unit)\cong\pi_0\flow^\bullet(\unit,\Omega^j\unit)\cong\pi_0\flow^\bullet(\Sigma^j\unit,\unit)\cong\Omega^\bullet_j(*);
    \end{equation}
the proposition follows.
\end{proof}

\begin{prop}\label{prop:sfrfr}
Let $\bbX$ be a framed flow category with finite object set and $\frakX^\bullet$ its associated spectrum $\flow^\bullet(\unit,\bbX)$. We have a homotopy equivalence 
    \begin{equation}
    \frakX^\sfr\simeq\frakX^\fr\wedge\frakR^\sfr.
    \end{equation}
\end{prop}

\begin{proof}
We wish to show the following diagram of stable $\infty$-categories and exact $\infty$-functors commutes:
    \begin{equation}
    \begin{tikzcd}
    \flow^\fr\arrow[r,"{\mathrm{forget}}"]\arrow[d,"{\flow^\fr(\unit,\cdot)}"] & \flow^\sfr\arrow[d,"{\flow^\sfr(\unit,\cdot)}"] \\
    \mathrm{Mod}_{\frakR^\fr}\arrow[r,"(\cdot)\wedge\frakR^\sfr"] & \mathrm{Mod}_{\frakR^\sfr}
    \end{tikzcd}
    .
    \end{equation}
Since $\flow^\fr$ is generated by $\unit$ (i.e., $\spectra\cong\flow^\fr$ is generated by $\bbS\cong\unit$) and each $\infty$-functor is colimit-preserving, it suffices to prove the statement for $\unit$. But now the statement follows by definition of $\frakR^\sfr$.
\end{proof}

We will require the following result.

\begin{thm}\label{thm:appendixb}
Let $\frakX^\fr$ be a finite spectrum. Suppose $H_*(\frakX;\bbZ)$ has a non-torsion element, then the map 
    \begin{equation}\label{eqn:appendixthm}
    \ho\mathrm{Mod}_{\frakR^\fr}(\frakX^\fr,\Sigma^j\frakR^\fr)\to\ho\mathrm{Mod}_{\frakR^\sfr}(\frakX^\sfr,\Sigma^j\frakR^\sfr),
    \end{equation}
where $\frakX^\sfr\equiv\frakX^\fr\wedge\frakR^\sfr$, is a non-zero map of abelian groups for infinitely many negative $j$.
\end{thm}

\begin{proof}
By Theorem \ref{thm:appendixa}, we may choose a sufficiently large prime $p$ such that the $p$-localization $\frakX^\fr_{(p)}$ becomes isomorphic to a direct sum of shifted Moore spectra in $\ho\spectra$; by assumption, this direct sum decomposition contains a $\bbS_{(p)}$ summand. In particular, we have that $\frakX^\sfr$ may be decomposed as a direct sum which contains a $\frakR^\sfr_{(p)}$ summand. The non-triviality of the map \eqref{eqn:appendixthm} is implied by the non-triviality of the map 
    \begin{equation}
    \pi^j\bbS_{(p)}\cong\ho\mathrm{Mod}_{\frakR^\fr}(\bbS_{(p)},\Sigma^j\frakR^\fr)\to\ho\mathrm{Mod}_{\frakR^\sfr}(\frakR^\sfr_{(p)},\Sigma^j\frakR^\sfr)
    \end{equation}
(i.e., we use the direct sum decomposition), but, by Proposition \ref{prop:appendixpoint}, this is simply the $p$-localization of the map 
    \begin{equation}
    \Omega^\fr_*(*)\to\Omega^\sfr_*(*)
    \end{equation}
which is a split-injection by Proposition \ref{lem:appendixinjective}. The theorem follows since $\bbS_{(p)}$ has infinitely many non-zero stable cohomotopy groups in negative degrees.
\end{proof}
\bibliography{References}{}
\bibliographystyle{alpha.bst}
\end{document}